\documentclass[12pt,reqno]{amsart}
\usepackage{amsthm, amsmath, amssymb, amscd, latexsym, multicol, verbatim, enumerate}
\usepackage[colorlinks=true, pdfstartview=FitV, linkcolor=blue, citecolor=blue, urlcolor=blue, breaklinks=true]{hyperref}

\usepackage[top=1.15in, bottom=1.15in, left=1.17in, right=1.17in]{geometry}
\usepackage{ytableau}
\usepackage{youngtab}
\usepackage{comment}
\usepackage{flexisym}
\usepackage{mathtools}
\usepackage{tikz-cd}
\usepackage{color}
\usepackage{hyperref}
\usepackage{amsmath}
\usepackage{amssymb}
\usepackage{amsfonts}
\usepackage{graphicx}
\usepackage[all]{xy}
\usepackage{verbatim}
\usepackage{mathdots}
\numberwithin{equation}{section}

\newtheorem{theorem}{Theorem}[section]
\newtheorem{corollary}[theorem]{Corollary}
\newtheorem{lemma}[theorem]{Lemma}
\newtheorem{proposition}[theorem]{Proposition}

\newtheorem{example}[theorem]{Example}
\newtheorem{definition}[theorem]{Definition}
\newtheorem{remark}[theorem]{Remark}

\newcommand{\n}{\mathfrak{n}}

\newcommand{\symplectic}{\mathfrak{sp}}
\newcommand{\h}{\mathfrak{h}}

\newcommand{\hv}{\nu_{\lambda}}

\newcommand{\V}{\textsc{V}}
\newcommand{\N}{\textsc{N}}

\newcommand{\X}{\textsc{X}}
\newcommand{\Lm}{\textsc{L}}
\newcommand{\Jm}{\textsc{J}}

\newcommand{\U}{\textsc{U}}

\newcommand{\I}{\textsc{I}}
\newcommand{\ideal}{\mathfrak{I}}
\newcommand{\tableau}{\textsc{T}}
\newcommand{\set}{\textsc{S}}
\newcommand{\Tset}{\textsf{T}}

\newcommand{\complete}{\textsc{Sp}\mathcal{F}}
\newcommand{\ov}{\overline}

\newcommand{\SyST}{\textsc{SyST}}

\title{Tropical symplectic flag varieties: a Lie-theoretic approach}
\author[G. Balla]{George Balla}
\address{Lehrstuhl f\"ur Algebra und Darstellungstheorie, RWTH Aachen University, Pontdriesch 10-16, 52062 Aachen, Germany}
\email{balla@art.rwth-aachen.de}
\author[X. Fang]{Xin Fang}
\address{Abteilung Mathematik, Department Mathematik/Informatik, Universit\"at zu K\"oln, 50931, Cologne, Germany}
\email{xinfang.math@gmail.com}

\begin{document}

\begin{abstract}
We study tropicalization of symplectic flag varieties with respect to the Pl\"ucker embedding. We identify a particular maximal prime cone in this tropicalization by explicitly giving its facets. For every interior point of this maximal cone, the corresponding Gr\"obner degeneration is the toric variety associated to the Feigin-Fourier-Littelmann-Vinberg (FFLV) polytope. Our main tool is the notion of birational sequences introduced by Fourier, Littelmann and the second author, which bridges between weighted PBW filtrations of representations of symplectic Lie algebras and degree functions on defining ideals of symplectic flag varieties.\\

\end{abstract}

\maketitle

\section{Introduction}
\subsection{Motivation}
In the past two decades tropical geometry has received great attention in enumerative geometry, mirror symmetry, Diophantine geometry, optimization, to name but a few. The tropical variety associated to an embedded projective variety is a pure sub-fan of the Gr\"obner fan of the defining ideal of the projective variety. The study of tropical varieties associated to geometric objects arising from Lie theory, such as Grassmannians and flag varieties, is initiated in the work of Speyer and Sturmfels \cite{SS04}. In this work, the tropical variety associated to the Grassmannian $\mathrm{Gr}(2,n)$ of $2$-planes in an $n$-space, with the Pl\"ucker embedding, gets a complete description by explicitly writing down defining inequalities of the maximal cones in the tropical variety using labelled trivalent trees. Such a complete description is not known for other Grassmannians.

The study of tropical varieties associated to flag varieties of type ${\tt A}$, called tropical flag varieties, started from the work of Bossinger, Lambogila, Mincheva and Mohammadi \cite{BLMM17}. In this work, tropical flag varieties of small rank were computed, and the authors showed existence of certain maximal cones related to certain string polytopes and FFLV polytopes using the method of toric degenerations. An explicit description of these cones by their facets or rays was not known in general.

Motivated by the work of Feigin, Fourier and Littelmann \cite{FFL11a, FFL11b} on PBW filtrations for simple Lie algebras and the work of Fourier, Reineke and the second author \cite{FFR16} on quantum PBW filtrations, weighted PBW filtrations for Lie algebras of type ${\tt A}_n$ are studied in \cite{FFFM19}: the cone consisting of such filtrations is shown to be sent bijectively to a maximal prime cone in the tropical flag variety of type ${\tt A}_n$. As a consequence, explicit facet description of this maximal prime cone was obtained. This result stands at the carrefour of certain directions in representation theory as follows.
\begin{enumerate}
\item[(a)] The Gr\"obner degenerations of the type ${\tt A}_n$ flag variety arising:
\begin{itemize}
\item[-] from points in the relative interior of this maximal prime cone yield the toric variety associated to the FFLV polytope \cite{FFFM19};
\item[-] from certain faces of this maximal prime cone are linear degenerate flag varieties introduced in \cite{CFFFR17, CFFFR20}.
\end{itemize}
\item[(b)] The cone consisting of weighted PBW filtrations of type ${\tt A}_n$ is closely related to the Auslander-Reiten theory for quiver representations: 
\begin{itemize}
\item[-] the facets of this cone correspond to Auslander-Reiten sequences in the category of representations of an equioriented type ${\tt A}_n$ quiver \cite{FFR21};
\item[-] the faces of this cone parametrize exact structures on the additive category consisting of finite dimensional representations of an equioriented type ${\tt A}_n$ quiver \cite{FG22}.
\end{itemize}
\end{enumerate}

We would like to note that Makhlin has provided in \cite{Mak22}, a full facet description of another maximal cone of the type {\tt A}$_n$ tropical flag variety using similar methods. This maximal cone also parametrizes Gr\"obner degenerations of the type {\tt A}$_n$ flag varieties, and in particular, every point in its relative interior corresponds to the toric variety associated with the Gelfand-Tsetlin polytope.

\subsection{Symplectic case}

In this paper, we generalize results in \cite{FFFM19} to the symplectic setting: we study weighted PBW degenerations of symplectic Lie algebras, their representations and corresponding geometry. This allows us to provide an explicit description of a maximal prime cone in the tropical symplectic flag variety. Since the notation has already been heavy, we will work with full symplectic flag varieties to avoid introducing further indices. Similar results hold for partial symplectic flag varieties, \emph{mutatis mutandis}. 

We define and treat the full symplectic flag variety $\complete_{2n}$ necessarily as the closure of a highest weight orbit for the action of the symplectic Lie group on its irreducible representations (see Definition \ref{Def:SymplecticFlag}). Consider the $2n$-dimensional complex vector space $\mathbb{C}^{2n}$ with a fixed symplectic form $\omega$. Then $\complete_{2n}$ coincides with the variety parametrizing length--$n$--flags of isotropic subspaces of $\mathbb{C}^{2n}$. The tropical symplectic flag variety $\mathrm{Trop}(\complete_{2n})$ is, up to a lineality space, a sub-fan of the Gr\"obner fan of the defining ideal of $\complete_{2n}$ with respect to the Pl\"ucker embedding; it is a pure polyhedral fan of dimension $n^2$ in $\mathbb{R}^{\mathcal{P}}$, where $\mathcal{P}$ is the index set of all Pl\"ucker coordinates on $\complete_{2n}$ (see Definition \ref{Def:SymplecticTropical}). A cone of $\mathrm{Trop}(\complete_{2n})$ is said to be a \emph{prime cone} if the corresponding initial ideal is prime. The following is one of our main results.

\begin{theorem}[Theorem \ref{Thm:main3}]\label{Thm:Main-Intro-1}
There is a maximal prime cone $\mathcal{C}_{2n}$ in the tropical symplectic flag variety $\mathrm{Trop}(\complete_{2n})$, with explicit facet description given in Lemma \ref{lem:maximalcone}, such that every point in its relative interior provides a Gr\"obner degeneration of $\complete_{2n}$ to the toric variety associated with the symplectic FFLV polytope (see Section \ref{Sec:SymplecticFFLV} for the definition).
\end{theorem}

The proof of Theorem \ref{Thm:Main-Intro-1} above uses the bridge between weighted PBW degeneration of the symplectic Lie algebra and degenerations of the symplectic flag variety coming from certain degree functions on its defining ideal with respect to the Pl\"ucker embedding.

The tropical variety associated to a special case of the symplectic flag variety, namely, the tropical symplectic Grassmannian, has been studied by Olarte and the first author in \cite{BO21}. Precisely, for $k\leq n$, the symplectic Grassmannian parametrizes $k$-dimensional isotropic linear subspaces of $\mathbb{C}^{2n}$. For $k=2$, it has been shown in \textit{loc.cit.} that there is a correspondence between maximal cones of the tropical symplectic Grassmannian and maximal cones of the usual tropical Grassmannian \cite[Prop. 5.1]{BO21}. Therefore, one has a full facet description of maximal cones of the tropical symplectic Grassmannian for this case. Such a description for higher $k$, and more generally, for arbitrary tropical symplectic flag varieties, was not yet known at the point of writing this paper.\\

Let $\mathfrak{g}=\mathfrak{sp}_{2n}(\mathbb{C})$ be the symplectic Lie algebra over $\mathbb{C}$, $\mathfrak{g}=\mathfrak{n}_+\oplus\mathfrak{h}\oplus\mathfrak{n}_-$ a triangular decomposition of $\mathfrak{g}$, and $\Phi^+$ the corresponding set of positive roots. 
We introduce a full-dimensional simplicial cone $\mathcal{K}_{2n}\subseteq\mathbb{R}^{\Phi^+}$. An element $\mathbf{d}\in\mathcal{K}_{2n}$ induces a Lie algebra filtration on $\mathfrak{n}_-$ by assigning for each $\beta\in\Phi^+$, degree $\mathbf{d}(\beta)$ to a chosen basis element $f_\beta$ of weight $-\beta$.
We denote the associated graded Lie algebra by $\mathfrak{n}_-^{\mathbf{d}}$. Let $\lambda$ denote a dominant integral weight for $\mathfrak{g}$. For a finite dimensional irreducible representation $\V_\lambda$ of $\mathfrak{g}$ with a highest weight vector $\nu_\lambda$, the filtration on the enveloping algebra $U(\mathfrak{n}_-)$ obtained from that on $\mathfrak{n}_-$ induces a filtration on $\V_\lambda$ (see Section \ref{Sec:Filtration}). We denote by $\V_\lambda^{\mathbf{d}}$, the associated graded space; it is a cyclic $U(\mathfrak{n}_-^{\mathbf{d}})$-module with cyclic vector $\nu_\lambda^{\mathbf{d}}$. The space $\V_\lambda^{\mathbf{d}}$ is said to be a \emph{weighted PBW degeneration} of $\V_\lambda$. For a multi-exponent $\mathbf{s}=(s_{\beta})_{\beta\in\Phi^+}$, we will denote by $f^{\mathbf{s}}_{\mathbf{d}}$ the class of $f^{\mathbf{s}}$ in $U(\n_{-}^{\mathbf{d}})$, where $f^{\mathbf{s}}:= \prod_{\beta\in \Phi^+}f_\beta^{s_{\beta}}$ with the root vectors $f_\beta$ ordered with respect to a fixed order. Let $\set(\lambda)$ denote the set of integral points in the symplectic FFLV polytope. We prove the following result which gives the compatibility of the monomial basis labelled by this polytope with weighted PBW degenerations.

\begin{theorem}[Theorem \ref{Thm:main1}]
For every point $\mathbf{d}\in\mathcal{K}_{2n}$, the set $\{f^{\mathbf{s}}_{\mathbf{d}}\cdot\nu^{\mathbf{d}}_{\lambda}\ |\  \mathbf{s}\in \set(\lambda)\}$ forms a basis of $\V_{\lambda}^{\mathbf{d}}$. 
\end{theorem}

For each $\mathbf{d}\in\mathcal{K}_{2n}$, we define the \emph{weighted degenerate symplectic flag variety} $\complete_{2n}^{\mathbf{d}}$ as the closure of the highest weight orbit $\overline{\exp(\mathfrak{n}_-^{\mathbf{d}})\cdot [\nu_\lambda^{\mathbf{d}}]}$ in the projective space $\mathbb{P}(\V_\lambda^{\mathbf{d}})$. We show in Theorem \ref{Thm:main2}, that the defining ideal of $\complete_{2n}^{\mathbf{d}}$ with respect to the Pl\"ucker embedding is the initial ideal of the defining ideal of the symplectic flag variety with respect to a degree function $\mathbf{w}_{\mathbf{d}}$ on the Pl\"ucker coordinates. The degree function $\mathbf{w}_{\mathbf{d}}$ induces an injective linear map $w:\mathbb{R}^{\Phi^+}\to\mathbb{R}^{\mathcal{P}}$ (See Section \ref{Sec:WeightedDeg}). Using this construction, we prove Theorem \ref{Thm:Main-Intro-1} by showing that the image of $\mathcal{K}_{2n}$ under $w$ is exactly the maximal prime cone $\mathcal{C}_{2n}$ in the tropical symplectic flag variety $\mathrm{Trop}(\complete_{2n})\subseteq\mathbb{R}^{\mathcal{P}}$.

By setting all entries of the point $\mathbf{d}\in\mathcal{K}_{2n}$ to 1, our degenerate varieties $\complete_{2n}^{\mathbf{d}}$ coincide the symplectic degenerate flag varieties $\complete_{2n}^a$ studied by Feigin, Finkelberg and Littelmann in \cite{FFL14}. This way, one also recovers the usual (non-weighted) PBW filtration on the symplectic Lie algebra and corresponding irreducible representations studied by Fourier, Feigin and Littelmann in \cite{FFL11b}. Just like the degenerate varieties in \emph{loc.cit.}, our degenerate varieties are flat degenerations (Theorem \ref{Thm:main2} and Corollary \ref{Cor:Basis}). Furthermore, we formulate and prove Borel-Weil type theorems in this setting as well (Theorems \ref{Thm:Borel1} and \ref{Thm:Borel2}).\\

Although the overall strategy of this paper is similar to that in \cite{FFFM19}, to deal with the symplectic case, we need to develop certain techniques to get control over the representation theory, geometry and related combinatorics.
\begin{enumerate}
\item[(a)] In the symplectic case, the fundamental representations are no longer minuscule, and they are only sub-representations of the exterior powers of the vector representation. We used a different argument in Section \ref{Subsec:Fund} to overcome the non-minuscule problem.
\item[(b)] Birational sequences introduced in \cite{FFL17} and the associated valuations are used to determine the defining ideal of the weighted degenerate symplectic flag varieties (Theorem \ref{Thm:main2}).
\item[(c)] We make use of the symplectic PBW-semistandard tableaux introduced in \cite{B20} and prove that they form a linear basis of the homogeneous coordinate ring of weighted degenerate symplectic flag variety with respect to the Pl\"ucker embedding.
\end{enumerate}
The point of view of \cite{FFFM19} and this work is different from that of \cite{BEZ21, BO21}: we treat flag varieties as homogeneous spaces associated to algebraic groups and make use of extra structures inherited from the group, while in \emph{loc.cit.}, the authors adopt and work with the linear algebra description of flag varieties, as flags of subspaces of the underlying vector spaces.\\

\subsection{Organisation of the paper} Section \ref{Sec:Intro} contains preliminaries on symplectic flag varieties. The main object of study, the tropical symplectic flag variety, is recalled in Section \ref{Sec:TropSymp}. In Section \ref{Sec:FFLVCone}, we treat weighted PBW degenerations of symplectic Lie algebras and their compatibility with FFLV bases. We deal with geometry of weighted PBW degenerations in section \ref{Sec:Geometry} and obtain compatibility with valuations coming from birational sequences. Lastly, in Section \ref{Sec:MaxCone}, we describe the facets of the maximal prime cone in the tropical symplectic flag variety using tools from the previous sections.

\subsection{Acknowledgements}
G.B. is grateful to Ghislain Fourier for fruitful discussions and continued support. He is funded by the Deutscher Akademischer Austauschdienst (DAAD) scholarship program: Research Grants - Doctoral Programs in Germany [Program ID 57440921]. This work is a contribution to the SFB-TRR 195 `Symbolic Tools in Mathematics and their Application' of the German Research Foundation (DFG).

\section{Preliminaries: symplectic flag varieties}\label{Sec:Intro}

For $k\in\mathbb{N}$ we denote $[k] :=\{1,2,\ldots,k\}$.

\subsection{Symplectic flag varieties}\label{Sec:SympFlg}

Let $\mathbb{C}^{2n}$ be a $2n$-dimensional complex vector space with standard basis $\{e_1,\ldots,e_{2n}\}$. For $i=1,\ldots,2n$, we will denote $\overline{i}:=2n+1-i$. This operation $\overline{\cdot}$ is an involution. The total order on $\mathbb{Z}$ induces the following order $1<2<\ldots<n<\overline{n}<\ldots<\overline{2}<\overline{1}$. If not mentioned otherwise, all vector spaces and Lie algebras are over $\mathbb{C}$. Let $\mathfrak{gl}_{2n}$ denote the Lie algebra on $2n\times 2n$ matrices with the commutator of two matrices as Lie bracket. The rows and columns of such a matrix will be indexed by $1,2,\ldots,n,\overline{n},\ldots,\overline{1}$.

Let $S$ be the square matrix of size $n$ with $1$ on the anti-diagonal and $0$ elsewhere. We set $J=\begin{pmatrix}0 & S\\ -S & 0\end{pmatrix}$ to be a square matrix of size $2n$. The symplectic Lie algebra $\mathfrak{sp}_{2n}$ is defined as a Lie-subalgebra of $\mathfrak{gl}_{2n}$:
$$\mathfrak{sp}_{2n}=\{X\in\mathfrak{gl}_{2n}\mid X^t J+JX=0\}.$$

The Lie-subalgebra $\mathfrak{h}$ consisting of diagonal matrices in $\mathfrak{sp}_{2n}$ is a Cartan subalgebra. We fix the following triangular decomposition: $\mathfrak{sp}_{2n}=\n_{+}\oplus  \h \oplus \n_{-}$ where $\n_+$ (resp. $\n_-$) consists of strictly upper-triangular (resp. strictly lower-triangular) matrices in $\mathfrak{sp}_{2n}$. The corresponding universal enveloping algebras will be denoted by $U(\n_+)$ and $U(\n_-)$.

Let $\Phi^{+}$ denote the corresponding set of positive roots and $\alpha_1,\ldots,\alpha_n$ be the simple roots for $\mathfrak{sp}_{2n}$. The positive roots can be divided into two sets namely:
\begin{align*}
\alpha_{i,j} &:= \alpha_i + \alpha_{i+1}+\ldots+\alpha_{j},\,\, \qquad\qquad\qquad\qquad 1\leq i\leq j\leq n;\\
\alpha_{i,\overline{j}} &:= \alpha_i + \alpha_{i+1}+\ldots+\alpha_{n}+\alpha_{n-1}+\ldots+\alpha_{j},\,\, 1\leq i\leq j < n.
\end{align*}
We will formally denote $\alpha_{i,\overline{n}}:=\alpha_{i,n}$. 


For each $\beta \in \Phi^{+}$, we fix a non-zero root vector $f_{\beta}\in (\n_{-})_{\beta}$ of weight $-\beta$ in the following way, where for $1\leq i\leq j\leq n$ we use the abbreviations $ f_{i,j}:=f_{\alpha_{i,j}}$ and $f_{i,\overline{j}}:=f_{\alpha_{i,\overline{j}}}$:
$$
\begin{array}{rcll}
f_{i,\overline{i}} &:=& E_{\overline{i},i}, & \text{ for } 1 \leq i \leq n,\\
f_{i,j} & := & E_{j+1,i}-E_{\overline{i},\overline{j+1}},& \text{ for } 1\leq i\leq j <n,\\
f_{i,\overline{j}} & := & E_{\overline{j},i}+E_{\overline{i},j}, & \text{ for } 1\leq i < j\leq n,
\end{array}
$$
where $E_{p,q}$ is the matrix  with zeros everywhere except for the entry 1 in the $p$-th row and $q$-th column.


Let $\{\omega_1,\ldots,\omega_n\}$ be the set of fundamental weights of $\symplectic_{2n}$ and $\Lambda^+:=\mathbb{N}\omega_1+\ldots+\mathbb{N}\omega_n$ be the monoid generated by them: elements in $\Lambda^+$ are dominant integral weights. For $\lambda= m_1\omega_1 +\ldots + m_n\omega_n\in\Lambda^+$, let $\V_{\lambda}$ denote the finite-dimensional irreducible $\symplectic_{2n}$-representation with highest weight $\lambda$. Fix a highest weight vector $\hv\in \V_{\lambda}$, we have: $\V_{\lambda}=U(\n_-)\cdot\hv$.  

Let $\N$ denote the simply connected Lie group  corresponding to $\n_-$. Then $\V_{\lambda}$ is a representation of $\N$, hence we have an action of $\N$ on the projectivization $\mathbb{P}(\V_{\lambda})$. Let $[\hv]$ denote the highest weight line through $\hv$ in $\mathbb{P}(\V_{\lambda})$.

Recall that a weight $\lambda=m_1\omega_1+\ldots+m_n\omega_n\in\Lambda^+$ is called regular, if $m_1,\ldots,m_n>0$.

\begin{definition}\label{Def:SymplecticFlag}
For a regular weight $\lambda\in\Lambda^+$, the complete symplectic flag variety is defined to be the orbit closure $\overline{\N\cdot [\hv]}$ in $\mathbb{P}(\V_{\lambda})$. 
\end{definition}

The isomorphism type of the projective variety $\overline{\N\cdot [\hv]}$ is independent of the choice of a regular weight $\lambda$: we will denote it simply by $\complete_{2n}$.

\subsection{Pl\"ucker embedding of symplectic flag varieties} 

We consider the fundamental representation $\V_{\omega_k}$ of $\mathfrak{sp}_{2n}$ with highest weight vector $\nu_{\omega_k}$. The representation $\V_{\omega_1}$ is the vector representation $\mathbb{C}^{2n}$ of $\mathfrak{sp}_{2n}$ with $\nu_{\omega_1}=e_1$. The fundamental representation $\V_{\omega_k}$ is a subrepresentation of the $k$-th exterior power of $\V_{\omega_1}$:
\begin{equation}\label{Eq:Embed}
\V_{\omega_k} \hookrightarrow {\bigwedge}^k \mathbb{C}^{2n}, \quad \nu_{\omega_k}\mapsto e_{1}\wedge\ldots \wedge e_{k}.
\end{equation}
For $\Jm=\{j_1,\ldots,j_k\} \subset \{1,\ldots,n ,\overline{n},\ldots,\overline{1}\}$ with $j_1<\ldots<j_k$, we denote  $e_{\Jm} :=e_{j_1}\wedge\ldots \wedge e_{j_k}$.

Let $\lambda=m_1\omega_1+\ldots+m_n\omega_n\in\Lambda^+$ be a regular weight. The irreducible representation $\V_\lambda$ can be realized as the Cartan component of the representation 
$$\U_{\lambda} =\V_{\omega_1}^{\otimes m_1}\otimes \ldots\otimes \V_{\omega_n}^{\otimes m_n}$$
via
$$\V_\lambda\hookrightarrow \U_\lambda,\ \ \nu_\lambda\mapsto u_\lambda:=
\nu_{\omega_1}^{\otimes m_1}\otimes \ldots\otimes \nu_{\omega_n}^{\otimes m_n}\in \U_{\lambda}.$$
This implies: $\complete_{2n}$ is isomorphic to $\overline{\N\cdot[u_\lambda]}\hookrightarrow\mathbb{P}(\U_\lambda)$.

We consider the following embedding of varieties
$$\mathbb{P}_n:=\mathbb{P}(\V_{\omega_1})\times\ldots\times \mathbb{P}(\V_{\omega_n})\xhookrightarrow{} \mathbb{P}(\V_{\omega_1})^{m_1}\times\ldots\times \mathbb{P}(\V_{\omega_n})^{m_n}\xhookrightarrow{} \mathbb{P}(\U_{\lambda})$$
where the first embedding is given by the diagonal embedding of $\mathbb{P}(\V_{\omega_i})$ into $\mathbb{P}(\V_{\omega_i})^{m_i}$, and the second one is the Segre embedding. According to the definition of the Segre embedding, $[u_\lambda]$ is the image of $([\nu_{\omega_1}],\ldots,[\nu_{\omega_n}])$ under the above embedding. It follows that for $x\in \mathrm{N}$, $[x\cdot u_\lambda]$ is the image of $([x\cdot \nu_{\omega_1}],\ldots,[x\cdot \nu_{\omega_n}])$ under the above embedding. We have therefore an embedding $\complete_{2n}\hookrightarrow \mathbb{P}_n$.

Composing with the embedding in (\refeq{Eq:Embed}), we obtain the Pl\"ucker embedding of $\complete_{2n}$:
\begin{equation}\label{Pluecker}
\complete_{2n}=\overline{\N\cdot[\hv]} \hookrightarrow{} \mathbb{P}_{n} \hookrightarrow{} \mathbb{P}\Big({\bigwedge}^1\mathbb{C}^{2n} \Big)\times \ldots\times \mathbb{P}\Big({\bigwedge}^n\mathbb{C}^{2n} \Big).
\end{equation}

\subsection{Defining relations}\label{Sec:DefRel}
We describe the defining ideal of $\complete_{2n}$ with respect to the Pl\"ucker embedding. As has been shown by de Concini in \cite{Dec79}, the defining ideal is generated by two kinds of relations: the usual quadratic Pl\"ucker relations, and certain linear relations.

Let $I=\{i_1,\ldots,i_d\}$ and $J=\{j_1,\ldots,j_d\}$ be two subsets of $\{1,\ldots,n\}$ with $i_1<\ldots<i_d$ and $j_1<\ldots<j_d$. The dominance partial order is defined by:
\begin{equation}\label{Dom}
I\leq J\text{ if and only if for any }1\leq k\leq d,\  \ i_k\leq j_k.
\end{equation}

For an ordered tuple $\Jm=(j_1,\ldots,j_d)$ with $1\leq j_1<\ldots<j_d\leq\overline{1}$, we set $\X_{\Jm}=\X_{j_1,\ldots,j_d}:=(e_{\Jm})^*\in \Big(\bigwedge^d\mathbb{C}^{2n} \Big)^*$ to be the corresponding homogeneous coordinate on $\mathbb{P}\Big(\bigwedge^d\mathbb{C}^{2n} \Big)$. Let $\mathcal{P}$ be the set of all such ordered tuples $\mathrm{J}$ of length $d$ for $d=1,\ldots,n$.  The (multi-)homogeneous coordinate ring of $\mathbb{P}\Big({\bigwedge}^1\mathbb{C}^{2n} \Big)\times \ldots\times \mathbb{P}\Big({\bigwedge}^n\mathbb{C}^{2n} \Big)$ is the polynomial ring $\mathcal{S}:=\mathbb{C}[\X_{\Jm}\mid \Jm\in\mathcal{P}]$.

The notation $\X_{\Jm}$ will be extended to all tuples by requiring for all $d=1,\ldots,n$ and $1\leq j_1<\ldots<j_d\leq \overline{1}$: 
$$\X_{j_{\sigma(1)},\ldots,j_{\sigma(d)}}:=(-1)^{\ell(\sigma)}\X_{j_1,\ldots,j_d}$$ 
where $\sigma\in\mathfrak{S}_d$ and $\ell(\sigma)$ is its inversion number.

We introduce some special elements in $\mathcal{S}$:
\begin{enumerate}
\item Let $\Lm, \Jm\subset \{1,\ldots,n,\overline{n},\ldots, \overline{1}\}$ be two tuples of lengths $p$ and $q$ respectively, where $1\leq q\leq p \leq n$. Suppose $\Lm=(l_1,\ldots,l_p)$ with $l_1<\ldots <l_p$, $\Jm =(j_1,\ldots,j_q)$ with $j_1<\ldots<j_q$ and $1\leq s\leq q$, we define the quadratic Pl\"ucker polynomial 
\begin{align}\label{eqn:relations-Cn}
   R_{\Lm,\Jm}^s:= \X_{\Lm} \X_{\Jm}- \sum_{1\leq r_1<\ldots<r_s\leq p}\X_{\Lm'}\X_{\Jm'},
\end{align} 
where 
$$\Lm'=(l_1,\ldots,l_{r_1-1},j_1,l_{r_1+1},\ldots,l_{r_s-1},j_s,l_{r_s+1},\ldots,l_p)$$
and
$$\Jm'=(l_{r_1},\ldots,l_{r_s},j_{s+1},\ldots,j_q).$$
\item We fix $1\leq k\leq n$ and let $\I_1=\{x_1,\ldots,x_\ell\}$, $\I_2=\{y_1,\ldots,y_{k-\ell}\}$ be subsets of $\{1,\ldots,n\}$ with $x_1<\ldots<x_\ell$ and $y_1<\ldots<y_{k-\ell}$. Let $\Gamma:=\I_1\cap \I_2$, $\Gamma=\{\gamma_1,\ldots,\gamma_t\}$ with $\gamma_1<\ldots<\gamma_t$ (note that $\Gamma$ could be the empty set). We define the tuple
$$(\I_1,\I_2):=(\gamma_1,\overline{\gamma_1},\ldots,\gamma_t,\overline{\gamma_t},a_1,\ldots,a_{\ell-t},\overline{b_{k-\ell-t}},\ldots,\overline{b_1})$$
where 
$$\I_1\setminus\Gamma=\{a_1,\ldots,a_{\ell-t}\}=:\tilde{\I}_1\text{ with }a_1<\ldots<a_{\ell-t},$$
$$\I_2\setminus\Gamma=\{b_1,\ldots,b_{k-\ell-t}\}=:\tilde{\I}_2\text{ with }b_1<\ldots<b_{k-\ell-t}.$$

Such a tuple $(\I_1,\I_2)$ is said to be \emph{reverse-admissible} if there exists a subset $\Tset\subset\{1,\ldots,n\}\setminus (\I_1\cup\I_2)$ with $|\Tset|=|\I_1\cap\I_2|$ and $\Tset< (\I_1\cap\I_2)$ (see \eqref{Dom} for the definition of this partial order). 

Notice that if $\I_1\cap \I_2=\emptyset$ then $(\I_1,\I_2)$ is reverse-admissible. Now assume that $(\I_1,\I_2)$ is not reverse-admissible with $\I_1\cap\I_2=\{\gamma_1,\ldots,\gamma_t\}=:\Gamma\neq\emptyset$. Choose $1\leq h_0\leq t$ to be minimal such that there exists a tuple $\Tset\subset \{1,\dots,n\}\setminus (\I_1\cup \I_2)$ of length $t-h_0$ with 
$ \Tset < (\gamma_{h_0+1},\dots,\gamma_t)$. Let $\Tset_{h_0+1}=\{\lambda_{h_0+1},\dots,\lambda_t\}$ be maximal (with respect to the partial order in \eqref{Dom}) among those $\Tset$. Finally we choose the maximal $b\in\{h_0+1,\dots,t\}$ such that $(\lambda_{h_0+1},\dots,\lambda_b) < (\gamma_{h_0},\dots,\gamma_{b-1})$,
or set $b=h_0$ if no such $b$ exists. Now set $\tilde\Gamma:=(\gamma_{h_0},\dots,\gamma_b)$ and $\mathsf{F}=\Gamma\backslash \tilde{\Gamma}$. We define a linear polynomial for non-reverse-admissible tuple $(\I_1,\I_2)$:
\begin{equation}\label{eqn:linear-relations}
    S_{(\I_1,\I_2)}:=\X_{(\I_1,\I_2)}-(-1)^{b-h_0+1}\sum_{\Gamma'}\X_{(\tilde{\I}_1\cup\mathsf{F}\cup \Gamma',\tilde{\I}_2\cup\mathsf{F}\cup \Gamma')},
\end{equation}
where the sum runs over those $\Gamma'$ satisfying $|\tilde{\Gamma}|=|\Gamma'|$ and $\Gamma'\cap (\I_1\cup\I_2)=\emptyset$.
\end{enumerate}

\begin{example}\label{Ex:sp4}
When $n=2$, the ideal $\mathfrak{I}_{4}$ is generated by the following polynomials:
$$\X_{12}\X_{\ov{2}}+\X_{2\ov{2}}\X_1-\X_{1\ov{2}}\X_2,\ \ \X_{1\ov{2}}\X_{\ov{1}}+\X_{\ov{2}\ov{1}}\X_1-\X_{1\ov{1}}\X_{\ov{2}},$$
$$\X_{2\ov{2}}\X_{\ov{1}}+\X_{\ov{2}\ov{1}}\X_2-\X_{2\ov{1}}\X_{\ov{2}},\ \ \X_{12}\X_{\ov{1}}+\X_{2\ov{1}}\X_1-\X_{1\ov{1}}\X_{2},$$
$$\X_{12}\X_{\ov{2}\ov{1}}-\X_{1\ov{2}}\X_{2\ov{1}}+\X_{1\ov{1}}\X_{2\ov{2}},\ \ \X_{1\ov{1}}+\X_{2\ov{2}}.$$
\end{example}

\begin{definition} 
Let $\ideal_{2n}$ denote the ideal in $\mathcal{S}$ generated by the quadratic polynomials $R_{\Lm,\Jm}^s$ for all possible $\Lm,\Jm,s$ and linear relations $S_{(\I_1,\I_2)}$ with non-reverse-admissible tuples $(\I_1,\I_2)$. 
\end{definition}

The following theorem is proved in \cite{Dec79} when the linear polynomials $S_{(\I_1,\I_2)}$ are defined for \emph{admissble} tuples. The above form for \emph{reverse-admissible} tuples is proved in \cite{B20}.

\begin{theorem}[\cite{Dec79, B20}]\label{Thm:DefIdeal}
The (prime) defining ideal of $\complete_{2n}$ with respect to the Pl\"ucker embedding \eqref{Pluecker} is precisely $\ideal_{2n}$.
\end{theorem}

The polynomial ring $\mathcal{S}$ is $\Lambda^+$-graded by assigning $\mathrm{deg}(\X_{\Jm})=\omega_{|\Jm|}$.

We will denote by $\mathbb{C}[\complete_{2n}]$ the (multi)-homogeneous coordinate ring of $\complete_{2n}$. According to Theorem \ref{Thm:DefIdeal}, $\mathbb{C}[\complete_{2n}]\cong \mathcal{S}/\ideal_{2n}$. Since both the ring $\mathcal{S}$ and the ideal $\ideal_{2n}$ are $\Lambda^+$-graded, so is $\mathbb{C}[\complete_{2n}]$. For a fixed $\lambda\in\Lambda^+$ we let $\mathbb{C}[\complete_{2n}]_\lambda$ denote the homogeneous component of degree $\lambda$. Moreover, in the following decomposition
$$\mathbb{C}[\complete_{2n}]=\bigoplus_{\lambda\in\Lambda^+}\mathbb{C}[\complete_{2n}]_\lambda,$$
each graded component $\mathbb{C}[\complete_{2n}]_\lambda$ is isomorphic to $\V_\lambda^*$ as $\mathfrak{sp}_{2n}$-modules \cite{Dec79}.

\section{Tropical symplectic flag varieties}\label{Sec:TropSymp}
In  this section, we consider the tropicalization of symplectic flag varieties using the defining ideal $\mathfrak{I}_{2n}\subseteq \mathcal{S}$ (according to Theorem \ref{Thm:DefIdeal}). The building blocks for tropical symplectic flag varieties, namely the tropical symplectic Grassmannians, have been studied by the second author and Olarte in \cite{BO21}, in the spirit of Speyer and Sturmfels' work on the tropical Grassmannian \cite{SS04}. 

In \cite{BO21}, it was shown that there are five equivalent characterizations of sympelctic Grassmannians, that turn out not to be equivalent tropically. This makes tropicalization in the symplectic setting subtle, nonetheless, a complete characterization was given. For the current paper, we choose the most natural construction out of the five characterizations for the tropical symplectic flag variety case, namely, the one given by taking the tropical variety of the tropicalization of the ideal $\mathfrak{I}_{2n}$, and we seek to understand its fan structure. For a comprehensive introduction to tropical geometry, we refer the reader to \cite{MS}.\\

Consider the Pl\"ucker embedding of $\complete_{2n}$ in \eqref{Pluecker}:
$$
\complete_{2n}\hookrightarrow{} \mathbb{P}\Big({\bigwedge}^1\mathbb{C}^{2n} \Big)\times \ldots\times \mathbb{P}\Big({\bigwedge}^n\mathbb{C}^{2n} \Big).
$$

For a polynomial 
$$f=\sum_{\mathbf{u}\in\mathbb{N}^{\mathcal{P}}}\lambda_{\mathbf{u}}X^{\mathbf{u}}\in\mathcal{S},$$ 
its initial form with respect to $\mathbf{v}\in\mathbb{R}^{\mathcal{P}}$ is defined to be
$$\mathrm{in}_{\mathbf{v}}(f):=\sum_{\mathbf{u}\cdot\mathbf{v}\text{ minimal}}\lambda_{\mathbf{u}}X^{\mathbf{u}}.$$
The initial ideal of $\mathfrak{I}_{2n}$ with respect to $\mathbf{v}\in\mathbb{R}^{\mathcal{P}}$ is defined as the ideal generated by the initial forms of all polynomial in $\mathfrak{I}_{2n}$ with respect to $\mathbf{v}$: 
$$\mathrm{in}_{\mathbf{v}}(\mathfrak{I}_{2n}):=\langle\mathrm{in}_{\mathbf{v}}(f)\mid f\in \mathfrak{I}_{2n}\rangle\subseteq \mathcal{S}.$$

From the above construction, it follows that every point $\mathbf{v}\in\mathbb{R}^{\mathcal{P}}$ defines a Gr\"obner degeneration of $\complete_{2n}$. Consider an equivalence relation on $\mathbb{R}^{\mathcal{P}}$ given by setting $\mathbf{v}\sim \mathbf{v}' $ whenever 
$$\mathrm{in}_{\mathbf{v}}(\mathfrak{I}_{2n}) = \mathrm{in}_{\mathbf{v}'}(\mathfrak{I}_{2n}).$$ Each equivalence class corresponds to points in the relative interior of a convex rational polyhedral cone in $\mathbb{R}^{\mathcal{P}}$. The collection of all such cones defines the Gr\"obner fan of $\complete_{2n}$. We are interested in the following subfan of this Gr\"obner fan.
 
\begin{definition}\label{Def:SymplecticTropical}
The tropical symplectic flag variety with respect to the Pl\"ucker embedding is defined by
$$\mathrm{Trop}(\complete_{2n}):=\{\mathbf{v}\in\mathbb{R}^{\mathcal{P}}\mid \mathrm{in}_{\mathbf{v}}(\mathfrak{I}_{2n})\text{ contains no monomials},\ \mathbf{v}([d])=0, d=1,\ldots,n\}.$$
\end{definition}

The normalization $\mathbf{v}([d])=0$ is chosen as a special section of the quotient map according to the multi-homogeneous property of the Pl\"ucker embedding.

By the Bieri-Groves theorem (\cite{BG84}), $\mathrm{Trop}(\complete_{2n})$ is a pure polyhedral fan in $\mathbb{R}^{\mathcal{P}}$ of dimension $n^2$. For example, $\mathrm{Trop}(\complete_{4})$ is a pure $4$-dimensional fan in $\mathbb{R}^{10}$ with a $2$-dimensional lineality space having $10$ rays and $15$ maximal cones (see \cite[Example 9.4]{BO21}).

The following notion is introduced in \cite{BLMM17}: a maximal cone in $\mathrm{Trop}(\complete_{2n})$ is called \emph{prime}, if the initial ideal associated to a point (hence any point) in its relative interior is a prime ideal. 

Giving explicit descriptions of all polyhedral cones in $\mathrm{Trop}(\complete_{2n})$ seems to be only possible when $n$ is very small. One of the goals of this paper is to prove the following theorem that generalizes results in the same direction for type ${\tt A}$ flag varieties in \cite{FFFM19}.

\begin{theorem}\label{Thm:main3} There is a maximal cone $\mathcal{C}_{2n}$ in the tropical symplectic flag variety $\mathrm{Trop}(\complete_{2n})$:
\begin{itemize}
    \item[(1)] that is a prime cone with explicit facet description given in Lemma \ref{lem:maximalcone}, and
    \item[(2)] such that every point in its relative interior provides a Gr\"obner degeneration of $\complete_{2n}$ to the toric variety associated with the symplectic FFLV polytope (see Section \ref{Sec:SymplecticFFLV} for the definition of this polytope).
\end{itemize}
\end{theorem}

The proof of Theorem \ref{Thm:main3} (1) will be given in Section \ref{Sec:MaxCone}, using weighted PBW degenerations of irreducible representations of the symplectic Lie algebra constructed in Section \ref{Sec:FFLVCone} and their bridge to valuations coming from birational sequences described in Section \ref{Sec:Geometry}. Part (2) of this theorem is proved in Section \ref{Sec:ProofThm(2)}.

\begin{remark}
    For type ${\tt A}$ flag varieties, such a maximal cone stands at the carrefour of tropical geometry, Lie theory (\cite{FFL11a, Fei12, FeFL17, FFFM19}), geometry of quiver representations (\cite{CFFFR17, CFFFR20}), an exact category and Hall algebras (\cite{FG22}). The connection between the maximal cone described in the current paper and the work of Boos and Cerulli Irelli \cite{BCI21} is not yet clear.
\end{remark}

\section{Weighted PBW degenerations and compatible bases}\label{Sec:FFLVCone}

The PBW degenerations of the symplectic Lie algebras and their representations are studied in \cite{FFL11b}. They proved that the symplectic FFLV basis (see Section \ref{Sec:SymplecticFFLV} below) is compatible with the filtration induced to the irreducible representations from the PBW filtration on the universal enveloping algebra $U(\mathfrak{n}_-)$. The goal of this section is to show that this basis is also compatible with a generalization of the PBW filtration that we consider.

\subsection{The symplectic FFLV basis}\label{Sec:SymplecticFFLV}
We follow \cite{FFL11b} to describe a basis of the irreducible representation $\V_\lambda$ parametrized by points in the FFLV polytope.

We first recall the notion of a symplectic Dyck path.

A \emph{symplectic Dyck path} is a sequence $\mathsf{p}=(p(0),\ldots,p(k))$, $k\geq 0$, of positive roots in $\Phi^+$ satisfying the following conditions:
\begin{enumerate}
    \item[(i)] the first root $p(0)= \alpha_i$ for some $1\leq i \leq n$, \emph{i.e.} it is simple;
    \item[(ii)] the last root $p(k)$ is either a simple root $p(k)=\alpha_l$ for some $1\leq l\leq n$ or $p(k)=\alpha_{j,\overline{j}}$ for some $1\leq j < n$;
    \item[(iii)] the elements in between satisfy the recursion rule: If $p(s)=\alpha_{p,q}\in\Phi^+$, then the next element in the sequence is either $p(s+1)=\alpha_{p,q+1}$ or $p(s+1)=\alpha_{p+1,q}$; where $q+1$ denotes the smallest element in $\{1,\ldots,n,\overline{n-1},\ldots,\overline{1}\}$ which is bigger than $q$. 
\end{enumerate}

Denote by $\mathbb{D}$ the set of all symplectic Dyck paths. For a dominant, integral weight $\lambda = \sum_{i=1}^n m_i \omega_i$, the symplectic FFLV polytope $\mathrm{FFLV}(\lambda)\subset \mathbb{R}^{\Phi^+}$ is the polytope defined by the inequalities $x_\beta\geq 0$ for $\beta\in\Phi^+$ and for all $\mathsf{p}=(p(0),\ldots,p(k))\in \mathbb{D}$:
\begin{align}
\begin{cases}
    & x_{p(0)} + \ldots + x_{p(k)} \leq m_i +\ldots + m_j, \quad \text{if} \quad p(0)=\alpha_i,\quad   p(k)=\alpha_j,\\
    & x_{p(0)} + \ldots + x_{p(k)} \leq m_i +\ldots + m_n,\quad \text{if} \quad p(0)=\alpha_i,\quad   p(k)=\alpha_{j,\overline{j}}.
    \end{cases}
\end{align}
We will denote by $\set(\lambda)$ the set of integral points in $\mathrm{FFLV}(\lambda)$.

Fix an enumeration of positive roots $\Phi^+=\{\beta_1,\ldots,\beta_N\}$. For a multi-exponent $\mathbf{s}=(s_{\beta})_{\beta\in\Phi^+}$, $s_{\beta}\in \mathbb{N}$, let $f^{\mathbf{s}}$ be the element
$$f^{\mathbf{s}} := f_{\beta_1}^{s_{\beta_1}}\cdots f_{\beta_N}^{s_{\beta_N}} \in U(\n_{-}).$$

Note that for $\lambda\in\Lambda^+$, we have fixed a highest weight vector $\hv \in \V_{\lambda}$.

\begin{theorem}[\cite{FFL11b}]\label{thm:FFLV-degenerate}
For any dominant integral weights $\lambda,\mu\in\Lambda^+$, the following hold true: 
\begin{enumerate}
\item  the Minkowski sum property:
\[\mathrm{FFLV}(\lambda + \mu)=\mathrm{FFLV}(\lambda)+\mathrm{FFLV}(\mu)\quad \text{and}\quad \set(\lambda+\mu)=\set(\lambda)+\set(\mu);\]
\item the elements $\{f^{\mathbf{s}}\cdot\hv\mid \mathbf{s}\in \set(\lambda)\}$ form a basis of $\V_{\lambda}$.
\end{enumerate}
\end{theorem}

The Minkowski sum property $\set(\lambda+\mu)=\set(\lambda)+\set(\mu)$ does not provide a unique decomposition of $\mathbf{s}\in\set(\lambda+\mu)$ into a sum of $\mathbf{s}_1\in\set(\lambda)$ and $\mathbf{s}_2\in\set(\mu)$. For $\lambda=m_1\omega_1+\ldots+m_n\omega_n$, we define the \emph{standard} decomposition of $\mathbf{s}\in\set(\lambda)$ into a sum of elements in 
$$\set(\omega_1)+\ldots+\set(\omega_1)+\ldots+\set(\omega_n)+\ldots+\set(\omega_n)$$ 
where $\set(\omega_k)$ appears $m_k$ times. Such a definition is given in an inductive manner on the height $\mathrm{ht}(\lambda):=m_1+\ldots+m_n$ of $\lambda$. Let $k:=\max\{\ell\mid m_\ell\neq 0\}$. We only need to decide how to decompose $\mathbf{s}$ into a sum $\mathbf{s}_1+\mathbf{s}_2$ with $\mathbf{s}_1\in\set(\omega_k)$ and $\mathbf{s}_2\in\set(\lambda-\omega_k)$.

For $\mathbf{s}=(s_\beta)\in\mathbb{N}^{\Phi^+}$ we denote $\mathrm{supp}\,\mathbf{s}:=\{\beta\in\Phi^+\mid s_\beta\neq 0 \}$ to be the support of $\mathbf{s}$. The set of positive roots admits a poset structure by defining $\beta_1\succ\beta_2$ if there exists a symplectic Dyck path from $\beta_1$ to $\beta_2$.  Let $\Phi_k^+\subseteq\Phi^+$ be the set of positive roots $\alpha_{i,j}$ satisfying $i\leq k\leq j$. We denote by $\mathbf{s}|_{\Phi_k^+}\subseteq\mathbb{N}^{\Phi^+}$ to be the restriction of $\mathbf{s}$ to $\Phi_k^+$ and then extended by zero to the entire $\Phi^+$.  We then define $\mathbf{s}_1\in\mathbb{N}^{\Phi^+}$ to be the characteristic function of the set of maximal elements (with respect to the above partial order) in $\mathrm{supp}\,(\mathbf{s}|_{\Phi_k^+})$.

\begin{lemma}\label{Lem:StdDec}
We have $\mathbf{s}_1\in\set(\omega_k)$ and $\mathbf{s}-\mathbf{s}_1\in\set(\lambda-\omega_k)$.
\end{lemma}

\begin{proof}
By \cite{ABS11}, the FFLV polytope $\mathrm{FFLV}(\lambda)$ is a marked chain polytope associated to the poset structure on $\Phi^+$ defined above with marking given by $\lambda$. When $\lambda=\omega_k$ is a fundamental weight, $\mathrm{FFLV}(\omega_k)$ is supported on $\Phi_k^+$; when considered inside $\mathbb{N}^{\Phi_k^+}$, it is the chain polytope associated to the induced poset $\Phi_k^+$. The lattice points in the chain polytope have a bijection to anti-chains in the poset \cite{Sta86}. The maximal element of a subset of $\Phi_k^+$ is an anti-chain, hence $\mathbf{s}_1\in\set(\omega_k)$.

Since a Dyck path intersects an anti-chain at at most one element, $\mathbf{s}-\mathbf{s}_1$ satisfies all defining inequalities of $\mathrm{FFLV}(\lambda-\omega_k)$, hence it is contained in $\set(\lambda-\omega_k)$.
\end{proof}

\subsection{Symplectic FFLV degree cone}

Keep the notations in Section \ref{Sec:SympFlg}. We define a polyhedral cone whose points will be used to construct a PBW type like filtrations on symplectic Lie algebras and their representations.

The vector space $\mathbb{R}^{\Phi^+}$ consists of functions from $\Phi^+$ to $\mathbb{R}$. For a point $\mathbf{d}\in\mathbb{R}^{\Phi^+}$, we will simply write $d_{i,j}:=\mathbf{d}(\alpha_{i,j})$ for $\alpha_{i,j}\in\Phi^+$. Notice that in our notation there is no difference between $d_{i,n}$ and $d_{i,\overline{n}}$.

\begin{definition}\label{def:weightsystem}
The symplectic FFLV degree cone for $\mathfrak{sp}_{2n}$ is the polyhedral cone $\mathcal{K}_{2n}\subseteq\mathbb{R}^{\Phi^+}$ defined by the following inequalities:
\begin{itemize}
    \item[(A$_i$)] $d_{i,i} + d_{i+1,i+1} \geq d_{i,i+1}, \qquad\qquad\ \text{for}\  1\leq i\leq n-1$;
    \item[(B$_{i,j}$)] $d_{i,j}   + d_{i+1,j+1} \geq d_{i,j+1}+d_{i+1,j}, \ \ \text{for}\  1\leq i< j \leq  n-1$;
    \item[(C$_{i,j}$)] $d_{i,\ov{j+1}} + d_{i+1,\ov{j}} \geq d_{i,\ov{j}} + d_{i+1,\ov{j+1}}, \ \ \text{for}\  1\leq i< j\leq  n-1$;
    \item[(D$_i$)] $2d_{i,\ov{i+1}} \geq d_{i,\ov{i}} + d_{i+1,\ov{i+1}},\qquad\qquad  \text{for}\  1\leq i\leq  n-1$.
\end{itemize}
That is to say, $\mathbf{d}\in\mathbb{R}^{\Phi^+}$ is contained in $\mathcal{K}_{2n}$ if and only if $\mathbf{d}$ satisfies the above inequalities.
\end{definition}

The polyhedral geometric properties of $\mathcal{K}_{2n}$ are summarized in the following proposition:

\begin{proposition}\label{Prop:Cone}
\begin{enumerate}
\item The cone $\mathcal{K}_{2n}$ is full dimensional in $\mathbb{R}^{\Phi^+}$.
\item The defining inequalities of $\mathcal{K}_{2n}$ give the complete set of its facets.
\item The cone $\mathcal{K}_{2n}$ has a lineality space of dimension $n$.
\item The cone $\mathcal{K}_{2n}$ is simplicial.
\end{enumerate}
\end{proposition}

\begin{proof}
We start with showing the full-dimensionality of $\mathcal{K}_{2n}$ by constructing a solution to the strict inequality system in the definition of $\mathcal{K}_{2n}$. First set $d_{i,i}=0$ for $1\leq i\leq n$. By ($\mathrm{A}_i$), the coordinates $d_{i,i+1}$ for $1\leq i\leq n-1$ are bounded above: choose a feasible solution for these coordinates. Then we look at the coordinates 
$$d_{1,3},\ d_{2,4},\ d_{1,4},\ \ldots,\ d_{n-2,n},\ \ldots,\ d_{1,n}$$
in this order. By ($\mathrm{B}_{i,j}$), a coordinate among the list above is bounded above once a feasible solution was chosen for all previous ones. Next we move to the coordinates
$$d_{n-1,\ov{n-1}},\ \ldots,\ d_{1,\ov{n-1}},\ \ldots,\ d_{2,\ov{2}},\ d_{1,\ov{2}},\ d_{1,\ov{1}}:$$
the same argument as above using ($\mathrm{C}_{i,j}$) and ($\mathrm{D}_i$) gives a point satisfying all strict inequalities. That is to say, we obtain a point in $\mathcal{K}_{2n}$ which has a small neighbourhood contained therein. This shows $\dim\mathcal{K}_{2n}=|\Phi^+|=n^2$.

We prove the other statements simultaneously. Let $L$ denote the lineality space of $\mathcal{K}_{2n}$. First notice that the linear space $L'$ defined by 
$$d_{i,j}-(d_{i,i}+\ldots+d_{j,j})=0,\ \ 1\leq i<j\leq n,$$
$$d_{i,\ov{j}}-(d_{i,i}+\ldots+d_{n,n}+\ldots+d_{j,j})=0,\ \ 1\leq i\leq j\leq n-1$$
is contained in $\mathcal{K}_{2n}$, hence $\dim L\geq n$. We set $d_{i,i}=0$ for $i=1,\ldots,n$. From the proof of the first part, once the values of $d_{i,i}$ are fixed, all other coordinates are bounded above. It follows that $\mathcal{K}_{2n}/L'$ is pointed and hence $L=L'$. This terminates the proof of (3). Now the number of inequalities in the definition of $\mathcal{K}_{2n}$ is $n(n-1)$, which coincides with the dimension of the pointed cone $\mathcal{K}_{2n}/L'$. This proves (4) and (2).
\end{proof}

The following inequalities can be deduced from the defining inequalities of $\mathcal{K}_{2n}$. We summarize them below for later use.

\begin{proposition}\label{Prop:generalinequalities}
For $\mathbf{d}\in\mathcal{K}_{2n}$, the following inequalities hold:
\begin{itemize}
	\item[($\mathrm{A}_{i,j,k}$)] $d_{i,j}   + d_{j+1,k} \geq d_{i,k},\ \ \text{for}\  1\leq i\leq j<k\leq n$;
    \item[($\mathrm{B}_{i,j,k,l}$)] $d_{i,j}   + d_{k,l} \geq d_{i,l}+d_{k,j},\ \ \text{for}\ 1\leq i< k\leq j<l\leq  n$;
    \item[($\mathrm{C}_{i,j,k,l}$)] $d_{i,\ov{l}} + d_{k,\ov{j}} \geq d_{i,\ov{j}} + d_{k,\ov{l}},\ \ \text{for}\ 1\leq i< k\leq j< l\leq  n$;
    \item[($\mathrm{D}_{i,j}$)] $2d_{i,\ov{j}} \geq d_{i,\ov{i}} + d_{j,\ov{j}},\ \ \text{for}\ 1\leq i< j\leq  n$; 
    \item[($\mathrm{E}_{i,j,k}$)] $d_{i,j} + d_{k,\ov{j+1}} \geq d_{i,\ov{k}}, \ \ \text{for}\ 1\leq i\leq k\leq j+1 \leq n$;
    \item[($\mathrm{F}_{i,j,k}$)] $d_{i,j}   + d_{j+1,\ov{k}} \geq d_{i,\ov{k}},\ \ \text{for}\  1\leq i\leq j<k\leq n$;
    \item[($\mathrm{G}_{i,j,k,l}$)] $d_{i,\ov{j}} + d_{k,\ov{l}} \geq d_{i,\ov{k}} + d_{j,\ov{l}},\ \ \text{for}\ 1\leq i< k\leq j< l\leq  n$;
    \item[($\mathrm{H}_{i,j,k,l}$)] $d_{i,\ov{l}} + d_{k,\ov{j}} \geq d_{i,\ov{k}} + d_{j,\ov{l}},\ \ \text{for}\ 1\leq i< k\leq j< l\leq  n$.
\end{itemize}
\end{proposition}

\begin{proof}
The inequalities ($\mathrm{A}_{i,j,k}$) can be deduced from the defining inequalities ($\mathrm{A}_{i}$) and ($\mathrm{B}_{i,j}$). The inequalities ($\mathrm{B}_{i,j,k,l}$) and ($\mathrm{C}_{i,j,k,l}$) follow from the defining inequalities ($\mathrm{B}_{i,j}$) and ($\mathrm{C}_{i,j}$). The inequalities ($\mathrm{D}_{i,j}$) are consequences of (C$_{i,j}$) and (D$_{i}$). The inequalities ($\mathrm{F}_{i,j,k}$) can be deduced from ($\mathrm{A}_{i}$), ($\mathrm{B}_{i,j}$) and ($\mathrm{C}_{i,j}$). The proofs are straightforward and are left to the reader. The inequality ($\mathrm{G}_{i,j,k,l}$) can be obtained by summing up the inequalities ($\mathrm{C}_{i,k,k,j}$), ($\mathrm{C}_{k,j,j,l}$) and $(-1)\times$($\mathrm{D}_{k,j}$). The inequality ($\mathrm{H}_{i,j,k,l}$) follows from ($\mathrm{C}_{i,j,k,l}$) and ($\mathrm{G}_{i,j,k,l}$).

It remains to show the inequality ($\mathrm{E}_{i,j,k}$). We will prove the harder one ($\mathrm{E}_{i,j,i}$). The proof of all other inequalities are similar. 

We proceed by descending induction on $i$. When $i=n-1$, the inequality reads 
$$(\mathrm{E}_{n-1,n-1,n-1}):\ d_{n-1,n-1}+d_{n-1,n}\geq d_{n-1,\ov{n-1}}.$$
It can be deduced by summing up the following two inequalities
$$(\mathrm{A}_{n-1}):\ d_{n-1,n-1}+d_{n,n}\geq d_{n-1,n},\ \ (\mathrm{D}_{n-1}):\ 2d_{n-1,n}\geq d_{n-1,\ov{n-1}}+ d_{n,n}.$$
Assume that ($\mathrm{E}_{i,j,i}$) holds for $i=k+1,\ldots,n-1$. We verify ($\mathrm{E}_{k,j,k}$) for a fixed $j=k,\ldots,n-1$. 

When $j=k$, the inequality ($\mathrm{E}_{k,k,k}$) is the sum of the inequalities ($\mathrm{D}_{k}$), ($\mathrm{C}_{k,k+1}$), $\ldots$, ($\mathrm{C}_{k,n-1}$), ($\mathrm{B}_{k,k+1}$), $\ldots$, ($\mathrm{B}_{k,n-1}$) and ($\mathrm{A}_{k}$). Now assume that $j\neq k$: the inequality ($\mathrm{E}_{k,j,k}$) is the sum of the inequalities 
($\mathrm{D}_{k}$), $2\times$($\mathrm{C}_{k,k+1}$), $\ldots$, $2\times$($\mathrm{C}_{k,j}$), ($\mathrm{C}_{k,j+1}$), ($\mathrm{C}_{k,n-1}$), ($\mathrm{B}_{k,j}$), $\ldots$, ($\mathrm{B}_{k,n-1}$) and ($\mathrm{E}_{k+1,j,k+1}$). By induction hypothesis, $\mathbf{d}\in\mathcal{K}_{2n}$ satisfies ($\mathrm{E}_{k+1,j,k+1}$), the proof is then complete.
\end{proof}

\begin{remark}\label{Rem:Strict}
Since in the proof we only sum up inequalities from the defining facets of $\mathcal{K}_{2n}$, if $\mathbf{d}$ is chosen from the interior of $\mathcal{K}_{2n}$, all inequalities in Proposition \ref{Prop:generalinequalities} are strict.
\end{remark}

\subsection{Filtrations from FFLV degree cone}\label{Sec:Filtration}
Points in $\mathcal{K}_{2n}$ give rise to filtrations on the Lie algebra $\mathfrak{n}_-$, the universal enveloping algebra $U(\mathfrak{n}_-)$ and the irreducible representation $\V_\lambda$.

We fix $\mathbf{d}\in\mathcal{K}_{2n}$. For $m\in\mathbb{R}$, we define a subspace of $\mathfrak{n}_-$:
$$(\n_{-})^{\mathbf{d}}_{\leq m}=\mathrm{span}\{f_{i,j}\mid\, 1\leq i\leq n,\ i\leq j\leq \overline{i},\  d_{i,j}\leq m\}$$
and the following subspace of $U(\mathfrak{n}_-)$:
$$U(\n_{-})^{\mathbf{d}}_{\leq m}=\mathrm{span}\{f_{i_1,j_1}\ldots f_{i_l,j_l}\mid  d_{i_1,j_1}+\ldots+d_{i_l,j_l}\leq m \}.$$

\begin{proposition}
For $\mathbf{d}\in\mathcal{K}_{2n}$, the subspaces $\{(\n_{-})^{\mathbf{d}}_{\leq m}\mid m\in\mathbb{R}\}$ define an $\mathbb{R}$-filtration of Lie algebra on $\n_-$; the subspaces $\{U(\n_{-})^{\mathbf{d}}_{\leq m}\mid m\in\mathbb{R}\}$ define an $\mathbb{R}$-filtration of algebra on $U(\n_-)$.
\end{proposition}

\begin{proof}
We need to verify that $[(\n_-)^{\mathbf{d}}_{\leq k},(\n_-)^{\mathbf{d}}_{\leq l}]\subseteq (\n_-)^{\mathbf{d}}_{\leq k+l}$. For $\alpha,\beta\in\Phi^+$, the Lie bracket $[f_\alpha,f_\beta]$ is non-zero if and only if $\alpha+\beta\in\Phi^+$. Written using the fixed basis for the weight spaces, the non-zero Lie brackets are precisely:
\begin{enumerate}
\item for $1\leq i\leq j\leq n-1$ and $j+1\leq k\leq \ov{i}$, $[f_{i,j},f_{j+1,k}]$ is a non-zero scalar multiple of $f_{i,k}$;
\item for $1\leq i\leq k\leq j+1\leq n$ with $i\neq j+1$, $[f_{i,j},f_{k,\ov{j+1}}]$ is a non-zero scalar multiple of $f_{i,\ov{k}}$.
\end{enumerate}
Since $\mathbf{d}\in\mathcal{K}_{2n}$, the inequalities ($\mathrm{A}_{i,j,k}$), ($\mathrm{E}_{i,j,k}$) and ($\mathrm{F}_{i,j,k}$) in Proposition \ref{Prop:generalinequalities} imply the desired inclusion. 

The statement on $U(\n_-)$ follows from the above argument.
\end{proof}

Similarly one defines the subspaces $(\n_{-})^{\mathbf{d}}_{< m}\subseteq (\n_{-})^{\mathbf{d}}_{\leq m}$ and $U(\n_{-})^{\mathbf{d}}_{<m}\subseteq U(\n_{-})^{\mathbf{d}}_{\leq m}$ by replacing the inequalities in the definitions of $(\n_{-})^{\mathbf{d}}_{\leq m}$ and $U(\n_{-})^{\mathbf{d}}_{\leq m}$ with the strict ones. We define the associated graded Lie algebra and the associated graded algebra as follows:
$$\n_-^{\mathbf{d}}:=\bigoplus_{m\in\mathbb{R}}(\n_-^{\mathbf{d}})_{m},\ \ \text{where}\ \ (\n_-^{\mathbf{d}})_{m}:=(\n_-)^{\mathbf{d}}_{\leq m}/(\n_-)^{\mathbf{d}}_{<m};$$
$$U(\n_-)^{\mathbf{d}}:=\bigoplus_{m\in\mathbb{R}}U(\n_-)^{\mathbf{d}}_{m},\ \ \text{where}\ \ U(\n_-)^{\mathbf{d}}_{m}:=U(\n_-)^{\mathbf{d}}_{\leq m}/U(\n_-)^{\mathbf{d}}_{<m}.$$
For a root vector $f_{i,j}$ with $1\leq i\leq n$ and $i\leq j\leq \ov{i}$, we denote $f_{i,j}^{\mathbf{d}}$ its class in $\mathfrak{n}_-^{\mathbf{d}}$.

The following proposition can be looked at as a kind of functoriality of the degeneration from $U(\n_-)$ to $U(\n_-^{\mathbf{d}})$.

\begin{proposition}\label{Prop:Isod}
The linear map $\mathfrak{n}_-^{\mathbf{d}}\to U(\mathfrak{n}_-)^{\mathbf{d}}$, sending $f_{i,j}^{\mathbf{d}}$ to the class of $f_{i,j}$ in $U(\mathfrak{n}_-)^{\mathbf{d}}$, induces an isomorphism of algebras, $U(\n_-^{\mathbf{d}})\cong U(\mathfrak{n}_-)^{\mathbf{d}}$.
\end{proposition}

\begin{proof}
From the inequalities (A), (E) and (F) in Proposition \ref{Prop:generalinequalities}, there exists a morphism of Lie algebras, $\n_-^{\mathbf{d}}\to U(\mathfrak{n}_-)^{\mathbf{d}}$. From the universal property, there exists a morphism of algebras, $U(\n_-^{\mathbf{d}})\to U(\mathfrak{n}_-)^{\mathbf{d}}$. It is an isomorphism by PBW theorem for $U(\mathfrak{n}_-)$.
\end{proof}

From now on we will not distinguish $U(\n_-^{\mathbf{d}})$ and $U(\mathfrak{n}_-)^{\mathbf{d}}$.

For a polyhedral cone $C$, let $\mathrm{relint}(C)$ denote its relative interior.

\begin{lemma}\label{Lem:Face}
Let $\mathbf{d},\mathbf{e}\in\mathcal{K}_{2n}$ be contained in the relative interior of the same face. 
\begin{enumerate}
\item The linear map 
$$\n_-^{\mathbf{d}}\to \n_-^{\mathbf{e}},\ \ (f_\beta)_{\mathbf{d}}\mapsto (f_\beta)_{\mathbf{e}}\ \ \text{for }\beta\in\Phi^+$$
is an isomorphism of Lie algebras.
\item The associated graded algebras $U(\n_-)^{\mathbf{d}}$ and $U(\n_-)^{\mathbf{e}}$ are isomorphic.
\end{enumerate}
\end{lemma}

\begin{proof}
The part (2) is a direct consequence of (1) and Proposition \ref{Prop:Isod}. We prove (1). Recall that whether the Lie bracket $[f_\alpha,f_\beta]$ for $\alpha,\beta\in\Phi^+$ in $\mathfrak{n}_-$ is zero depends on whether $\alpha+\beta$ is a positive root. Given $\mathbf{d}\in\mathcal{K}_{2n}$, the Lie brackets in $\mathfrak{n}_-^{\mathbf{d}}$ are determined by whether inequalities $(\mathrm{A}_{i,j,k})$, $(\mathrm{E}_{i,j,k})$ and $(\mathrm{F}_{i,j,k})$ in Proposition \ref{Prop:generalinequalities} are strict or not. Notice that in the proof of the proposition, the above inequalities are obtained from summing up certain defining inequalities of $\mathcal{K}_{2n}$ in Definition \ref{def:weightsystem}. From Proposition \ref{Prop:Cone}, they are facets of $\mathcal{K}_{2n}$. Therefore the Lie brackets in $\mathfrak{n}_-^{\mathbf{d}}$ will remain the same when $\mathbf{d}$ varies in the relative interior of a face.
\end{proof}

As a consequence, the isomorphism type of the Lie algebra $\n_-^{\mathbf{d}}$ is constant on the relative interior of each face of $\mathcal{K}_{2n}$.

We have a closer look at $U(\n_-)^{\mathbf{d}}$ and $\V_\lambda^{\mathbf{d}}$ for $\mathbf{d}\in\mathrm{relint}(\mathcal{K}_{2n})$.

\begin{lemma}\label{Lem:DegAlg}
If $\mathbf{d}\in\mathrm{relint}(\mathcal{K}_{2n})$, $U(\n_-)^{\mathbf{d}}$ is isomorphic to the symmetric algebra $S(\n_-)$ as an algebra.
\end{lemma}

\begin{proof}
The same argument as in the proof of Lemma \ref{Lem:Face} shows that if all inequalities in Definition \ref{def:weightsystem} are strict, then for any $\alpha,\beta\in\Phi^+$, $[f_\alpha,f_\beta]=0$ in $\mathfrak{n}_-^{\mathbf{d}}$, hence $\mathfrak{n}_-^{\mathbf{d}}$ is an abelian Lie algebra.
\end{proof}

Such a filtration on $U(\n_-)$ induces a filtration on cyclic modules. For any dominant integral weight $\lambda$, the simple $\symplectic_{2n}$-module $\V_{\lambda}=U(\n_-)\cdot \nu_{\lambda}$ is cyclic. For a fixed $\mathbf{d}\in\mathcal{K}_{2n}$, we consider the induced $\mathbb{R}$-filtration 
\[(\V_{\lambda})_{\leq m}^{\mathbf{d}}=U(\n_-)^{\mathbf{d}}_{\leq m}\cdot\nu_{\lambda}\]
and similarly its subspace $(\V_{\lambda})_{< m}^{\mathbf{d}}$.
Let us denote the associated graded space by $\V_{\lambda}^{\mathbf{d}}$, i.e.
\[\V_{\lambda}^{\mathbf{d}} = \bigoplus_{m\geq 0} (\V_{\lambda}^{\mathbf{d}})_{m},\ \ \text{where}\ \ (\V_{\lambda}^{\mathbf{d}})_{m} = (\V_{\lambda})^{\mathbf{d}}_{\leq m}/(\V_{\lambda})^{\mathbf{d}}_{<m}.\] 

The vector space $\V_{\lambda}^{\mathbf{d}}$ carries naturally a graded $U(\n_{-}^{\mathbf{d}})$-module structure. Indeed, for any $k,l\in \mathbb{R}$, we have by definition
\[U(\n_-)_{\leq k}(\V_{\lambda})_{\leq l}\subseteq (\V_{\lambda})_{\leq k+l}.\]
The $U(\mathfrak{n}_-^{\mathbf{d}})$-module $\V_\lambda^{\mathbf{d}}$ will be termed a \textit{weighted PBW degeneration} of $\V_{\lambda}$. 

Let $\nu_{\lambda}^{\mathbf{d}}$ denote the image of $\nu_{\lambda}$ in $\V_{\lambda}^{\mathbf{d}}$. The $U(\n_{-}^{\mathbf{d}})$-module $\V_{\lambda}^{\mathbf{d}}$ is cyclic, having $\nu_{\lambda}^{\mathbf{d}}$ as a cyclic vector.

We fix $\mathbf{d}\in\mathrm{relint}(\mathcal{K}_{2n})$. For $\lambda\in\Lambda^+$, since $\V_\lambda^{\mathbf{d}}$ is a cyclic $U(\n_-)^{\mathbf{d}}$-module, we obtain a surjective $U(\n_-)^{\mathbf{d}}$-module map
$$\varphi_\lambda^{\mathbf{d}}:S(\n_-)\to \V_\lambda^{\mathbf{d}},\ \ x\mapsto x\cdot \nu_\lambda^{\mathbf{d}}.$$
We set $\mathrm{I}_\lambda^{\mathbf{d}}:=\ker\varphi_\lambda^{\mathbf{d}}$. It is an ideal in $S(\n_-)$.

Recall that in Section \ref{Sec:SymplecticFFLV} we have fixed an enumeration of positive roots $\Phi^+=\{\beta_1,\ldots,\beta_N\}$. For a multi-exponent $\mathbf{s}=(s_{\beta})_{\beta\in\Phi^+}$, we will denote by $f^{\mathbf{s}}_{\mathbf{d}}$ the class of $f^{\mathbf{s}}$ in $U(\n_{-})^{\mathbf{d}}$.

The first main result of the paper is the following compatibility of the FFLV basis and the weighted PBW degenerations.

\begin{theorem}\label{Thm:main1}
For every point $\mathbf{d}\in\mathcal{K}_{2n}$, the set $\{f^{\mathbf{s}}_{\mathbf{d}}\cdot\nu^{\mathbf{d}}_{\lambda}\ |\  \mathbf{s}\in \set(\lambda)\}$ forms a basis of $\V_{\lambda}^{\mathbf{d}}$. 
\end{theorem}

The rest of this section will be devoted to the proof of the theorem.

\subsection{Compatibility of FFLV basis: fundamental representations}\label{Subsec:Fund}

In this section we start from proving Theorem \ref{Thm:main1} for $\mathbf{d}\in\mathrm{relint}(\mathcal{K}_{2n})$ and $\lambda=\omega_k$, then we will explain how to adapt the proof to deal with those $\mathbf{d}$ in a proper face of $\mathcal{K}_{2n}$.

\subsubsection{Vector representation}
Assume that $\mathbf{d}\in\mathrm{relint}(\mathcal{K}_{2n})$. We study the vector representation $\V_{\omega_1}$. The actions of root vectors on $e_l\in\V_{\omega_1}$ with $1\leq l\leq\ov{1}$ are given by:
\begin{enumerate}
\item for $1\leq i\leq j<n$, \quad $f_{i,j}\cdot e_l=\delta_{i,l}\,e_{j+1}-\delta_{\ov{j+1},l}\,e_{\ov{i}}$;
\item for $1\leq i\leq n$, \qquad\quad $f_{i,\ov{i}}\cdot e_l=\delta_{i,l}\,e_{\ov{i}}$;
\item for $1\leq i< j\leq n$, \quad $f_{i,\ov{j}}\cdot e_l=\delta_{i,l}\,e_{\ov{j}}+\delta_{j,l}\,e_{\ov{i}}$.
\end{enumerate}

For $1\leq i<j\leq\ov{1}$, set 
$$M_i^j:=\{\mathbf{s}\in\mathbb{N}^{\Phi^+}\mid f^{\mathbf{s}}\cdot e_i=\pm e_j\}$$
and consider the following function 
$$\mathfrak{d}^{\mathbf{d}}:\mathbb{N}^{\Phi^+}\to\mathbb{R},\ \ \mathbf{s}\mapsto \mathbf{d}\cdot\mathbf{s}.$$
The set $M_i^j$ is non-empty: from the action of the root vectors, the element $\mathbf{s}_{i,j}$ such that 
$$
f^{\mathbf{s}_{i,j}}=\begin{cases} 
f_{i,j-1}, & \text{if } i,j\leq n; \\
f_{i,j}, & \text{if } i\leq n,\ \overline{n}\leq j\leq\overline{i}; \\
f_{\overline{j},\overline{i}} & \text{if } i\leq n,\ j >\overline{i}; \\
f_{\overline{j},\overline{i}-1}, & \text{if } i,j\geq \overline{n};
\end{cases}
$$
is contained in $M_i^j$.

\begin{lemma}\label{Lem:Fund1}
The function $\mathfrak{d}^{\mathbf{d}}$ has a unique minimum on $M_i^j$ attained at $\mathbf{s}_{i,j}$.
\end{lemma}

\begin{proof}
We proceed by induction on $j-i$. When $j=i+1$, the set $M_i^j$ contains only the element $\mathbf{s}_{i,j}$, and there is nothing to prove.

In general, take $\mathbf{s}=(s_\beta)_{\beta\in\Phi^+}\in M_i^j$ and denote $|\mathbf{s}|:=\sum_{\beta\in\Phi^+}s_\beta$. 

Assume furthermore that $\mathfrak{d}^{\mathbf{d}}$ attains minimum at $\mathbf{s}$. If $|\mathbf{s}| \geq 2$, there must be a root vector in the monomial $f^{\mathbf{s}}$ sending $e_i$ to a non-zero multiple of some $e_k$ with $i<k<j$. The remaining part of this monomial has the form $f^{\mathbf{t}}$ with $\mathbf{t}\in M_k^j$. Since $j-k<j-i$, by induction, in order to minimize $\mathfrak{d}^{\mathbf{d}}$, $\mathbf{t}$ must be $\mathbf{s}_{k,j}$. This implies $|\mathbf{s}|=2$ and it has the following form
\begin{equation}\label{Eq:Mins}
f^{\mathbf{s}}=\begin{cases} 
f_{i,k-1}f_{k,j-1}, & \text{if } j\leq n; \\
f_{i,k-1}f_{k,j}, & \text{if } k\leq n,\ \overline{n}\leq j\leq\overline{k}; \\
f_{i,k-1}f_{\ov{j},\ov{k}}, & \text{if } k\leq n,\ j>\overline{k}; \\
f_{i,k}f_{\overline{j},\overline{k}-1}, & \text{if } i\leq n,\ \overline{n}\leq k\leq\overline{i}; \\
f_{\ov{k},\ov{i}}f_{\overline{j},\overline{k}-1}, & \text{if } i\leq n,\ k>\overline{i}; \\
f_{\overline{k},\overline{i}-1}f_{\overline{j},\overline{k}-1}, & \text{if } i\geq \overline{n}.
\end{cases}
\end{equation}
Since $\mathbf{d}$ is chosen from the relative interior of $\mathcal{K}_{2n}$, by Remark \ref{Rem:Strict}, the inequalities in Proposition \ref{Prop:generalinequalities} are strict. According to the inequalities ($\mathrm{A}_{i,k-1,j-1}$), ($\mathrm{F}_{i,k-1,\overline{j}}$),  ($\mathrm{E}_{i,k-1,\overline{j}}$), ($\mathrm{E}_{i,k,\overline{j}}$),  ($\mathrm{F}_{\overline{j},\overline{k}-1,i}$) and ($\mathrm{A}_{\overline{j},\overline{k}-1,\overline{i}-1}$) respectively, the monomials in \eqref{Eq:Mins} can not minimize $\mathfrak{d}^{\mathbf{d}}$.

We have shown that any $|\mathbf{s}|\geq 2$ can not minimize $\mathfrak{d}^{\mathbf{d}}$, therefore $|\mathbf{s}|=1$ and $f^{\mathbf{s}}$ has to be $f^{\mathbf{s}_{i,j}}$.
\end{proof}

\begin{proposition}
   For $\mathbf{d}\in\mathrm{relint}(\mathcal{K}_{2n})$, the set $\{f^{\mathbf{s}}_{\mathbf{d}}\cdot\nu^{\mathbf{d}}_{\omega_1}\ |\  \mathbf{s}\in \set(\omega_1)\}$ forms a basis for $\V_{\omega_1}^{\mathbf{d}}$.
\end{proposition}
\begin{proof}
    By \cite[Lemma 4.1]{FFL11b}, non-zero elements in $\set(\omega_1)$ are precisely $\{\mathbf{s}_{1,j}\mid 2\leq j\leq\overline{1}\}$. Now Lemma \ref{Lem:Fund1} implies that $\{f^{\mathbf{s}}_{\mathbf{d}}\cdot\nu^{\mathbf{d}}_{\omega_1}\ |\  \mathbf{s}\in \set(\omega_1)\}$ spans $\V_{\omega_1}^{\mathbf{d}}$. By Theorem \ref{thm:FFLV-degenerate}, $|\set(\omega_1)|=\dim\V_{\omega_1}^{\mathbf{d}}$, hence the above generating set forms indeed a basis of $\V_{\omega_1}^{\mathbf{d}}$. 
\end{proof}

\subsubsection{General case}\label{Sec:GenCase}

We prove Theorem \ref{Thm:main1} for $\mathbf{d}\in\mathrm{relint}(\mathcal{K}_{2n})$ and an arbitrary fundamental representation $V_{\omega_k}$. The representation $V_{\omega_k}$ is a subrepresentation of ${\bigwedge}^k\mathbb{C}^{2n}$ generated by the highest weight vector $\nu_{\omega_k}:= e_1\wedge\ldots\wedge e_k$.

For a subset $I=\{i_1,\ldots,i_k\}$ of $\{1,\ldots,n,\overline{n},\ldots,\overline{1}\}$ with $i_1<\ldots<i_k$, we denote $e_I:=e_{i_1}\wedge\ldots\wedge e_{i_k}\in{\bigwedge}^k\mathbb{C}^{2n}$ and $e_I^*$ its dual basis element. For another subset $J=\{j_1,\ldots,j_k\}$ of $\{1,\ldots,n,\overline{n},\ldots,\overline{1}\}$ with $j_1<\ldots <j_k$, we define
$$M_I^J:=\{\mathbf{s}\in\mathbb{N}^{\Phi^+}\mid e_J^*(f^{\mathbf{s}}\cdot e_I)\neq 0\}.$$

We describe an element in $M_I^J$ to show that it is non-empty. We set $K=I\cap J$, $P=\{p_1,\ldots,p_s\}=I\setminus K$ with $p_1<\ldots<p_s$ and $Q=\{q_1,\ldots,q_s\}=J\setminus K$ with $q_1<\ldots<q_s$. The element 
\begin{equation}\label{Eq:SIJ}
\mathbf{s}_{I,J}:=\mathbf{s}_{p_1,q_s}+\ldots+\mathbf{s}_{p_s,q_1}
\end{equation}
is contained in $M_I^J$. In fact we have $\mathbf{s}_{I,J}\in M_{P}^Q$.

\begin{lemma}\label{Lem:Fundk}
The function $\mathfrak{d}^{\mathbf{d}}$ has a unique minimum on $M_I^J$ attained at $\mathbf{s}_{I,J}$.
\end{lemma}

\begin{proof}
We proceed by induction on $k$, the index of the fundamental representation. The starting point $k=1$ is Lemma \ref{Lem:Fund1}.

We first show that it suffices to prove the lemma under the assumption $I\cap J=\emptyset$. 

If this intersection is not empty, there exists $r$ and $m$ such that $i_r=j_m$. Let $\mathbf{m}$ be a minimum of $\mathfrak{d}^{\mathbf{d}}$ on $M_I^J$. We show that the monomial $f^{\mathbf{m}}$ does not contain root vectors of the form $f^{\mathbf{s}_{i_r,t}}$ for $t>i_r$ nor root vectors of the form $f^{\mathbf{s}_{t,j_m}}$ for $t<j_m$. If it contains for example $f^{\mathbf{s}_{i_r,t}}$, it would also contain $f^{\mathbf{s}_{p,j_m}}$ for some $p<j_m$ because $i_r=j_m$. But by Lemma \ref{Lem:Fund1}, $\mathfrak{d}^{\mathbf{d}}$ would take smaller value at $f^{\mathbf{s}_{p,t}}$. Now the element 
$$\mathbf{m'}:=\mathbf{m}-\mathbf{s}_{i_r,t}-\mathbf{s}_{p,j_m}+\mathbf{s}_{p,t}$$ 
is still in $M_I^J$ with $\mathfrak{d}^{\mathbf{d}}(\mathbf{m}')<\mathfrak{d}^{\mathbf{d}}(\mathbf{m})$. This contradicts to the choice of $\mathbf{m}$. Similar argument shows the statement on $f^{\mathbf{s}_{t,j_m}}$.  As a consequence, such an element $\mathbf{m}$ is in fact contained in $M_P^Q$ with $P=I\setminus (I\cap J)$ and $Q=J\setminus (I\cap J)$. Since $|P|<k$, we can proceed by induction.

From now on we assume that $I\cap J=\emptyset$. From the action on the tensor product, for $\mathbf{t}\in M_I^J$, $f^{\mathbf{t}}\cdot e_I$ is a linear combination of $f^{\mathbf{t}_1}\cdot e_{i_1}\wedge\ldots \wedge f^{\mathbf{t}_k}\cdot e_{i_k}$ with $\mathbf{t}_1+\ldots+\mathbf{t}_k=\mathbf{t}$. Since $\mathbf{t}\in M_I^J$, there exists at least one collection $\{\mathbf{t}_1,\ldots,\mathbf{t}_k\}$ such that $f^{\mathbf{t}_1}\cdot e_{i_1}\wedge\ldots \wedge f^{\mathbf{t}_k}\cdot e_{i_k}$ is proportional to $e_J$. By Lemma \ref{Lem:Fund1}, in order to minimize $\mathfrak{d}^{\mathbf{d}}$, we have to choose those $\mathbf{t}_1,\ldots,\mathbf{t}_k$ with $|\mathbf{t}_r|\leq 1$ for $r=1,\ldots,k$. The assumption $I\cap J=\emptyset$ implies $|\mathbf{t}_1|=\ldots=|\mathbf{t}_k|=1$. Again, to minimize $\mathfrak{d}^{\mathbf{d}}$, it is necessary that $\mathbf{t}_r$ has the form $\mathbf{s}_{i_r,j}$ where the index $j$ is uniquely determined by: $f^{\mathbf{t}_r}\cdot e_{i_r}=\pm e_j$.

We reformulate the above observation using symmetric group. For any $\sigma\in \mathfrak{S}_k$ we associate to it an element 
$$\mathbf{s}_\sigma:=\mathbf{s}_{i_1,j_{\sigma(1)}}+\ldots+\mathbf{s}_{i_k,j_{\sigma(k)}}\in\mathbb{N}^{\Phi^+}.$$
From definition, $\mathbf{s}_\sigma\in M_I^J$. The above argument shows that a minimum of $\mathfrak{d}^{\mathbf{d}}$ can not be attained outside of the set $\{\mathbf{s}_\sigma\mid\sigma\in\mathfrak{S}_k\}$.

We claim that if $\sigma>\sigma'$ in the Bruhat order of $\mathfrak{S}_k$, then $\mathfrak{d}^{\mathbf{d}}(\mathbf{s}_\sigma)<\mathfrak{d}^{\mathbf{d}}(\mathbf{s}_{\sigma'})$. This will terminate the proof since $\mathbf{s}_{I,J}=\mathbf{s}_{w_0}$ where $w_0$ is the unique maximal element in $\mathfrak{S}_k$ sending $i$ to $k+1-i$.

It remains to prove the claim. Keeping $\mathbf{d}$ in mind, we will simply write $s_{i,j}:=\mathfrak{d}^{\mathbf{d}}(\mathbf{s}_{i,j})$ and 
$$s_\sigma:=\sum_{r=1}^k s_{i_r,j_{\sigma(r)}}.$$
It suffices to consider the case when $\sigma>\sigma'$ is a covering relation in the Bruhat poset. There exists therefore a permutation $\sigma_{p,q}$, $1\leq p<q\leq k$, swapping $p$ and $q$, such that $\sigma=\sigma'\sigma_{p,q}$. In this case,
$$s_\sigma-s_{\sigma'}=s_{i_p,j_{\sigma'(q)}}+s_{i_q,j_{\sigma'(p)}}-s_{i_p,j_{\sigma'(p)}}-s_{i_q,j_{\sigma'(q)}}.$$
Notice that from $\sigma>\sigma'$, $p<q$ and $I\cap J=\emptyset$ it follows $i_p<i_q<j_{\sigma'(p)}<j_{\sigma'(q)}$.

There are several cases to consider:
\begin{enumerate}
\item[(i)] $j_{\sigma'(q)}\leq n$. In this case, the strict inequality ($\mathrm{B}_{i_p,i_q,j_{\sigma'(p)}-1,j_{\sigma'(q)}-1}$) in Proposition \ref{Prop:generalinequalities} gives $s_\sigma-s_{\sigma'}<0$.
\item[(ii)] $i_p\geq\overline{n}$. Let $\alpha:=\overline{j}_{\sigma'(q)}$, $\beta:=\overline{j}_{\sigma'(p)}$, $\gamma:=\overline{i}_{q}-1$ and $\delta:=\overline{i}_{p}-1$. Then $\alpha<\beta\leq\gamma<\delta$ and the strict inequality ($\mathrm{B}_{\alpha,\gamma,\beta,\delta}$) in Proposition \ref{Prop:generalinequalities} gives $s_\sigma-s_{\sigma'}<0$.
\item[(iii)] $i_q\leq n$ and $j_{\sigma'(p)}\geq\overline{n}$. We execute a case-by-case analysis:
\begin{enumerate}
\item[(iii.1)] We consider the case when $j_{\sigma'(q)}>\overline{i}_p$.
\begin{itemize}
\item If $\overline{i}_q<j_{\sigma'(p)}\leq \overline{i}_p$, it follows from the strict inequality ($\mathrm{H}_{\overline{j}_{\sigma'(q)},\overline{j}_{\sigma'(p)},i_p,i_q}$) in Proposition \ref{Prop:generalinequalities} that $s_\sigma-s_{\sigma'}<0$.
\item If $j_{\sigma'(p)}\leq \overline{i}_q$, it follows from the strict inequality ($\mathrm{G}_{\overline{j}_{\sigma'(q)},i_q,i_p,\overline{j}_{\sigma'(p)}}$) in Proposition \ref{Prop:generalinequalities} that $s_\sigma-s_{\sigma'}<0$.
\item If $j_{\sigma'(q)}> \overline{i}_p$, it follows from the strict inequality ($\mathrm{C}_{\overline{j}_{\sigma'(q)},i_p,\overline{j}_{\sigma'(p)},i_q}$) in Proposition \ref{Prop:generalinequalities} that $s_\sigma-s_{\sigma'}<0$.
\end{itemize}
\item[(iii.2)] We consider the case when $j_{\sigma'(q)}\leq\overline{i}_p$.
\begin{itemize}
\item If $j_{\sigma'(p)}> \overline{i}_q$, it follows from the strict inequality ($\mathrm{G}_{i_p,\overline{j}_{\sigma'(p)},\overline{j}_{\sigma'(q)},i_q}$) in Proposition \ref{Prop:generalinequalities} that $s_\sigma-s_{\sigma'}<0$.
\item If $j_{\sigma'(p)}\leq \overline{i}_q$ and $j_{\sigma'(q)}> \overline{i}_q$, it follows from the strict inequality ($\mathrm{H}_{i_p,i_q,\overline{j}_{\sigma'(q)},\overline{j}_{\sigma'(p)}}$) in Proposition \ref{Prop:generalinequalities} that $s_\sigma-s_{\sigma'}<0$.
\item If $j_{\sigma'(q)}< \overline{i}_q$, it follows from the strict inequality ($\mathrm{C}_{i_p,\overline{j}_{\sigma'(q)},i_q,\overline{j}_{\sigma'(p)}}$) in Proposition \ref{Prop:generalinequalities} that $s_\sigma-s_{\sigma'}<0$.
\end{itemize}
\end{enumerate}
\item[(iv)] $j_{\sigma'(p)}\leq n$ and $j_{\sigma'(q)}\geq\overline{n}$. This and the next cases are similar but simpler than (iii), the verifications are left to the reader.
\item[(v)] $i_p\leq n$ and $i_q\geq \overline{n}$. See (iv).
\end{enumerate}

The proof is then complete.
\end{proof}


\begin{proposition}
    For $\mathbf{d}\in\mathrm{relint}(\mathcal{K}_{2n})$, the set $\{f_{\mathbf{d}}^{\mathbf{s}}\cdot\nu_{\omega_k}^{\mathbf{d}}\mid \mathbf{s}\in \set(\omega_k)\}$ forms a basis for $\V_{\omega_k}^{\mathbf{d}}$.
\end{proposition}
\begin{proof}
    The set of lattice points $\set(\omega_k)$ in the FFLV polytope for a fundamental weight $\omega_k$ is described in \cite[Lemma 4.1]{FFL11b}. It follows that for any $J\subseteq\{1,\ldots,n,\overline{n},\ldots,\overline{1}\}$ of cardinality $k$, the element $\mathbf{s}_{[k],J}$ is contained in $\set(\omega_k)$. Since $\V_{\omega_k}$ is a subrepresentation of ${\bigwedge}^k\mathbb{C}^{2n}$, it follows
$$\set(\omega_k)=\{\mathbf{s}_{[k],J}\mid J\subseteq\{1,\ldots,n,\overline{n},\ldots,\overline{1}\}\text{ with }|J|=k\}.$$

Let $\mathbf{s}\in \mathbb{N}^{\Phi^+}$ with $f^{\mathbf{s}}\cdot \nu_{\omega_k}\neq 0$. We choose $J\subseteq\{1,\ldots,n,\overline{n},\ldots,\overline{1}\}$ such that  $e_J^*(f^{\mathbf{s}}\cdot \nu_{\omega_k})\neq 0$ when considered as an element of ${\bigwedge}^k\mathbb{C}^{2n}$. If $\mathbf{s}\neq \mathbf{s}_{[k],J}$, by Lemma \ref{Lem:Fundk}, $\mathfrak{d}^{\mathbf{d}}:M_{[k]}^J\to\mathbb{R}$ does not attain its minimum at $\mathbf{s}$. This implies $f_{\mathbf{d}}^{\mathbf{s}}\cdot \nu_{\omega_k}^{\mathbf{d}}=0$, and therefore the set $\{f_{\mathbf{d}}^{\mathbf{s}}\cdot\nu_{\omega_k}^{\mathbf{d}}\mid \mathbf{s}\in \set(\omega_k)\}$ spans $\V_{\omega_k}^{\mathbf{d}}$. By Theorem \ref{thm:FFLV-degenerate} (2), $|\set(\omega_k)|=\dim\V_{\omega_k}$, hence the above set forms a basis of $\V_{\omega_k}^{\mathbf{d}}$.
\end{proof}

\subsubsection{From interior to boundary}\label{Sec:Boundary}
We now prove Theorem \ref{Thm:main1} for $\mathbf{d}\in\mathcal{K}_{2n}$ and $\lambda=\omega_k$.

\begin{proposition}
    For any $\mathbf{d}\in  \mathcal{K}_{2n}$, the set $\{f_{\mathbf{d}}^{\mathbf{s}}\cdot\nu_{\omega_k}^{\mathbf{d}}\mid \mathbf{s}\in \set(\omega_k)\}$ forms a basis for $\V_{\omega_k}^{\mathbf{d}}$.
\end{proposition}
\begin{proof}
When $\mathbf{d}$ is chosen from the boundary of the cone $\mathcal{K}_{2n}$, in the proofs of Lemma \ref{Lem:Fund1} and Lemma \ref{Lem:Fundk} above, the strict inequalities used therein from Proposition \ref{Prop:generalinequalities} are not necessarily strict, and therefore the proofs imply that the function $\mathfrak{d}^{\mathbf{d}}$ attains one of its minimum at $\mathbf{s}_{I,J}$ (such a minimum is not necessarily unique). The set $\{f_{\mathbf{d}}^{\mathbf{s}}\cdot\nu_{\omega_k}^{\mathbf{d}}\mid \mathbf{s}\in \set(\omega_k)\}$ still spans $\V_{\omega_k}^{\mathbf{d}}$, hence they form a basis by Theorem \ref{thm:FFLV-degenerate} (2).
\end{proof}

\subsection{Proof of Theorem \ref{Thm:main1}}

Let $\mathbf{d}\in\mathcal{K}_{2n}$. We proceed the proof by induction on the height $\mathrm{ht}(\lambda)=\sum_{k=1}^{n}m_k$ of the weight $\lambda=\sum_{k=1}^{n}m_k\omega_k$. When $\mathrm{ht}(\lambda)=1$, $\lambda$ is a fundamental weight: such cases are settled in Section \ref{Subsec:Fund}. For the inductive step we need to describe the Cartan component  of the $U(\mathfrak{n}_-^{\mathbf{d}})$-module $\V_\lambda^{\mathbf{d}}\otimes\V_\mu^{\mathbf{d}}$. We start from the following lemma:

\begin{lemma}\label{Lem:indep}
The set $\{f^{\mathbf{s}}_{\mathbf{d}}\cdot (\nu_\lambda^{\mathbf{d}}\otimes\nu_\mu^{\mathbf{d}})\mid \mathbf{s}\in\set(\lambda+\mu)\}$ is linearly independent in $\V_\lambda^{\mathbf{d}}\otimes\V_\mu^{\mathbf{d}}$.
\end{lemma}

\begin{proof}
The proof is the same as \cite[Proposition 4]{FFR16}, we sketch it for the convenience of the reader.

Assume that there exists a non-trivial linear combination 
\begin{equation}\label{Eq:Combo}
\sum_{\mathbf{s}\in\set(\lambda+\mu)} c_{\mathbf{s}}f_{\mathbf{d}}^{\mathbf{s}}\cdot(\nu_\lambda^{\mathbf{d}}\otimes \nu_\mu^{\mathbf{d}})=0.
\end{equation}
Let $\succ^{\mathbf{d}}$ be the partial order on $\mathbb{N}^{\Phi^+}$ defined by: $\mathbf{s}\succ^{\mathbf{d}}\mathbf{t}$ if and only if $\mathfrak{d}^{\mathbf{d}}(\mathbf{s})>\mathfrak{d}^{\mathbf{d}}(\mathbf{t})$. We fix a linearization of this partial order, which is also denoted by $\succ^{\mathbf{d}}$.

Let $\mathbf{s}:=\max\{\mathbf{t}\in\set(\lambda+\mu)\mid c_{\mathbf{t}}\neq 0\}$ where the maximum is taken with respect to $\succ^{\mathbf{d}}$. By induction hypothesis we can write
\begin{equation}\label{Eq:tensor}
f_{\mathbf{d}}^{\mathbf{s}}\cdot(\nu_\lambda^{\mathbf{d}}\otimes \nu_\mu^{\mathbf{d}})=\sum_{\mathbf{t}\in\set(\mu),\mathbf{s}-\mathbf{t}\in\set(\lambda)}\gamma_{\mathbf{t}}f_{\mathbf{d}}^{\mathbf{s}-\mathbf{t}}\cdot\nu_\lambda^{\mathbf{d}}\otimes f_{\mathbf{d}}^{\mathbf{t}} \cdot\nu_\mu^{\mathbf{d}}
\end{equation}
with $\gamma_{\mathbf{t}}\in\mathbb{C}$. We set $\mathbf{t}_0:=\min\{\mathbf{t}\in\set(\mu)\mid\gamma_{\mathbf{t}}\neq 0\}$ where the minimum is taken with respect to $\succ^{\mathbf{d}}$. In \eqref{Eq:tensor} there exists a term $\gamma_{\mathbf{t}_0}f_{\mathbf{d}}^{\mathbf{s}-\mathbf{t}_0}\cdot\nu_\lambda^{\mathbf{d}}\otimes f_{\mathbf{d}}^{\mathbf{t}_0} \cdot\nu_\mu^{\mathbf{d}}$ with $\gamma_{\mathbf{t}_0}\neq 0$. By the maximality of $\mathbf{s}$ and minimality of $\mathbf{t}_0$, such a term appears only once in the linear combination \eqref{Eq:Combo}, contradicting to the assumption that $\gamma_{\mathbf{t}_0}\neq 0$.

As a consequence of this contradiction, there exists no such a maximal $\mathbf{s}$, implying that all the coefficients in \eqref{Eq:Combo} are zero.
\end{proof}

\begin{proposition}\label{Prop:Cartan}
For $\mathbf{d}\in\mathcal{K}_{2n}$ and $\lambda,\mu\in\Lambda^+$, the Cartan component $U(\mathfrak{n}_-^{\mathbf{d}})\cdot (\nu_\lambda^{\mathbf{d}}\otimes \nu_\mu^{\mathbf{d}})$ of $\V_\lambda^{\mathbf{d}}\otimes\V_\mu^{\mathbf{d}}$ is isomorphic to $\V_{\lambda+\mu}^{\mathbf{d}}$ as $U(\mathfrak{n}_-^{\mathbf{d}})$-modules.
\end{proposition}

\begin{proof}
We denote by $W$ to be the Cartan component and consider the $U(\mathfrak{n}_-^{\mathbf{d}})$-module morphism $\V_{\lambda+\mu}^{\mathbf{d}}\to W$ determined by $\nu_{\lambda+\mu}^{\mathbf{d}}\mapsto\nu_\lambda^{\mathbf{d}}\otimes \nu_\mu^{\mathbf{d}}$. 

We first show that this morphism is surjective, hence $\dim W\leq\dim \V_{\lambda+\mu}$. It suffices to prove that if $f\in U(\mathfrak{n}_-^{\mathbf{d}})$ such that $f\cdot \nu_{\lambda+\mu}^{\mathbf{d}}=0$, then $f\cdot (\nu_\lambda^{\mathbf{d}}\otimes \nu_\mu^{\mathbf{d}})=0$. Assume that $f\in U(\mathfrak{n}_-)^{\mathbf{d}}_k$, from $f\cdot \nu_{\lambda+\mu}^{\mathbf{d}}=0$, there exists $F\in U(\mathfrak{n}_-)$ such that $F\cdot v_{\lambda+\mu}=0$, and $F$ admits a decomposition $F=F_1+F_2$ such that the class of $F_1$ in $U(\mathfrak{n}_-^{\mathbf{d}})$ is $f$, and $F_2\in U(\mathfrak{n}_-)^{\mathbf{d}}_{<k}$. As $\V_{\lambda+\mu}$ is the Cartan component of $\V_\lambda\otimes\V_\mu$ as $U(\mathfrak{n}_-)$-modules, it follows $F\cdot (\nu_\lambda\otimes\nu_\mu)=0$. Note that $\V_\lambda^{\mathbf{d}}\otimes\V_\mu^{\mathbf{d}}$ is the associated graded space of the canonical filtration on $\V_\lambda\otimes\V_\mu$ whose  component of degree $\leq m$ is 
$$\bigoplus_{s+t= m} (\V_\lambda)^{\mathbf{d}}_{\leq s}\otimes (\V_\mu)^{\mathbf{d}}_{\leq t}.$$
It then follows that the class of $F\cdot (\nu_\lambda\otimes\nu_\mu)$ in $\V_\lambda^{\mathbf{d}}\otimes\V_\mu^{\mathbf{d}}$ is $f\cdot (\nu_\lambda^{\mathbf{d}}\otimes\nu_\mu^{\mathbf{d}})$, which is therefore zero.

By Theorem \ref{thm:FFLV-degenerate} (2), it remains to apply Lemma \ref{Eq:Combo} to conclude.
\end{proof}

Now we can complete the proof of Theorem \ref{Thm:main1}. By Proposition \ref{Prop:Cartan}, the set $\{f_{\mathbf{d}}^{\mathbf{s}}\cdot\nu_{\lambda+\mu}^{\mathbf{d}}\mid \mathbf{s}\in \set(\lambda+\mu)\}$ is sent to a linearly independent set in $\V_\lambda^{\mathbf{d}}\otimes\V_\mu^{\mathbf{d}}$. The theorem thus follows from Theorem \ref{thm:FFLV-degenerate} (2).

\subsection{Monomial ideal}

When $\mathbf{d}$ is taken from the interior of $\mathcal{K}_{2n}$, we have the following 

\begin{corollary}\label{Cor:Monomial}
For $\mathbf{d}$ in the interior of $\mathcal{K}_{2n}$ and $\lambda\in\Lambda^+$, if $\mathbf{s}\notin \set(\lambda)$, then $f^{\mathbf{s}}_{\mathbf{d}}\cdot \nu_{\lambda}^{\mathbf{d}}=0$. In particular, the ideal $\mathrm{I}_\lambda^{\mathbf{d}}$ is monomial.
\end{corollary}

\begin{proof}
We again proceed by induction on the height $\mathrm{ht}(\lambda)$. When $\mathrm{ht}(\lambda)=1$, the corollary is a consequence of Lemma \ref{Lem:Fundk}. 

We consider the weight $\lambda+\mu$ for $\lambda,\mu\in\Lambda^+$. If $\mathbf{s}\notin\set(\lambda+\mu)$, for any decomposition $\mathbf{s}=\mathbf{s}_1+\mathbf{s}_2$ where $\mathbf{s}_1,\mathbf{s}_2\in\mathbb{N}^{\Phi^+}$, either $\mathbf{s}_1\notin\set(\lambda)$ or $\mathbf{s}_2\notin\set(\mu)$. It follows by induction that $f^{\mathbf{s}}_{\mathbf{d}}\cdot (\nu_{\lambda}^{\mathbf{d}}\otimes\nu_\mu^{\mathbf{d}})=0$ in $\V_\lambda^{\mathbf{d}}\otimes \V_\mu^{\mathbf{d}}$. By Proposition \ref{Prop:Cartan}, $\V_{\lambda+\mu}^{\mathbf{d}}$ is the Cartan component in $\V_\lambda^{\mathbf{d}}\otimes \V_\mu^{\mathbf{d}}$, hence $f^{\mathbf{s}}_{\mathbf{d}}\cdot \nu_{\lambda+\mu}^{\mathbf{d}}=0$.
\end{proof}

\section{Geometry of weighted PBW degenerations}\label{Sec:Geometry}

\subsection{Weighted degenerate symplectic flag varieties}\label{Sec:WeightedDeg}

In this section we fix $\mathbf{d}\in\mathcal{K}_{2n}$. We first introduce a geometric object associated to weighted PBW degenerate module $\V_\lambda^{\mathbf{d}}$ for $\lambda=m_1\omega_1+\ldots+m_n\omega_n\in\Lambda^+$.

Since the Lie algebra $\mathfrak{n}_-^{\mathbf{d}}$ is nilpotent, the exponential map is well-defined, and $\mathrm{N}^{\mathbf{d}}:=\exp(\mathfrak{n}_-^{\mathbf{d}})$ is a connected simply connected Lie group with Lie algebra $\mathfrak{n}_-^{\mathbf{d}}$.

\begin{definition}
For $\mathbf{d}\in\mathcal{K}_{2n}$ we define the \emph{weighted degenerate symplectic flag variety} by
$$\complete_{2n}^{\mathbf{d}}:=\overline{\mathrm{N}^{\mathbf{d}}\cdot [\nu_\lambda^{\mathbf{d}}}]\subseteq \mathbb{P}(\V_{\lambda}^{\mathbf{d}}).$$
\end{definition}

By Proposition \ref{Prop:Cartan}, as $U(\mathfrak{n}_-^{\mathbf{d}})$-modules, 
$$\V_\lambda^{\mathbf{d}}\hookrightarrow \mathrm{U}_\lambda^{\mathbf{d}}:=(\V_{\omega_1}^{\mathbf{d}})^{\otimes m_1}\otimes\ldots\otimes(\V_{\omega_n}^{\mathbf{d}})^{\otimes m_n}$$
is the Cartan component. Then the same argument as in the construction of the embedding in \eqref{Pluecker} can be applied here to show that the image of the following embedding
$$\complete_{2n}^{\mathbf{d}}\hookrightarrow \mathbb{P}(\V_{\lambda}^{\mathbf{d}})\hookrightarrow \mathbb{P}(\U_{\lambda}^{\mathbf{d}})$$
is in fact contained in $\mathbb{P}_n^{\mathbf{d}}:=\mathbb{P}(\V_{\omega_1}^{\mathbf{d}})\times\ldots\times \mathbb{P}(\V_{\omega_n}^{\mathbf{d}})$, embedded in $\mathbb{P}(\U_{\lambda}^{\mathbf{d}})$ via diagonal and Segre embeddings:
$$\mathbb{P}_n^{\mathbf{d}}\hookrightarrow\mathbb{P}(\V_{\omega_1}^{\mathbf{d}})^{m_1}\times\ldots\times \mathbb{P}(\V_{\omega_n}^{\mathbf{d}})^{m_n}\hookrightarrow\mathbb{P}(\mathrm{U}_\lambda^{\mathbf{d}}).$$
As a summary, we fix the following Pl\"ucker embedding of $\complete_{2n}^{\mathbf{d}}$:
$$\complete_{2n}^{\mathbf{d}}\hookrightarrow\mathbb{P}_n^{\mathbf{d}}\hookrightarrow\mathbb{P}\Big({\bigwedge}^1\mathbb{C}^{2n} \Big)\times \ldots\times \mathbb{P}\Big({\bigwedge}^n\mathbb{C}^{2n} \Big).$$

To distinguish this degenerate setup for different $\mathbf{d}$, we denote $\mathcal{S}^{\mathbf{d}}:=\mathbb{C}[\X_{\Jm}^{\mathbf{d}}\mid \Jm\in\mathcal{P}]$ (see Section \ref{Sec:DefRel} for the definition of $\mathcal{P}$): the defining ideal of $\complete_{2n}^{\mathbf{d}}$ with respect to the above Pl\"ucker embedding will be denoted by $\mathfrak{I}_{2n}^{\mathbf{d}}\subseteq \mathcal{S}^{\mathbf{d}}$. Notice that giving $\X_{\Jm}^{\mathbf{d}}$ degree $\omega_{|J|}$, the ideal $\mathfrak{I}_{2n}^{\mathbf{d}}$ is $\Lambda^+$-graded. The goal of this section is to show that the ideal $\mathfrak{I}_{2n}^{\mathbf{d}}$ is in fact an initial ideal of $\mathfrak{I}_{2n}$ with respect to a weight vector.

We define a weight map 
$$w:\mathcal{K}_{2n}\to\mathbb{R}^{\mathcal{P}},\ \ \mathbf{d}\mapsto\mathbf{w}_{\mathbf{d}},$$
where $\mathbf{w}_{\mathbf{d}}$ is the function on $\mathcal{P}$ sending $J$ to $-\mathfrak{d}^{\mathbf{d}}(\mathbf{s}_{[k],J})$ (see \eqref{Eq:SIJ} for the definition of $\mathbf{s}_{[k],J}$). The function $\mathbf{w}_{\mathbf{d}}$ induces gradings on $\mathcal{S}$ and $\mathcal{S}^{\mathbf{d}}$ by assigning degree $\mathbf{w}_{\mathbf{d}}(J)$ to $X_J$ and $X_J^{\mathbf{d}}$.

When the element $\mathbf{d}\in\mathcal{K}_{2n}$ is clear from context, we will drop it from $\mathbf{w}_{\mathbf{d}}$ and simply write $\mathbf{w}$.

\begin{theorem}\label{Thm:main2}
The ideal $\mathfrak{I}_{2n}^{\mathbf{d}}$ coincides with the initial ideal $\mathrm{in}_{\mathbf{w}}(\mathfrak{I}_{2n})$.
\end{theorem}

\begin{remark}
When $\mathbf{d}\in\mathbb{R}^{\Phi^+}$ is the constant function with $d(\beta)=1$, the variety $\complete_{2n}^{\mathbf{d}}$ is the symplectic degenerate flag variety introduced by Feigin, Finkelberg and Littelmann in \cite{FFL14}.
\end{remark}

\subsection{Symplectic PBW-semistandard tableaux}\label{Sec:PBWTab}
We recall a set of tableaux from \cite{B20} which is compatible with weighted degenerations (Corollary \ref{Cor:Basis}), and which will be useful in our constructions henceforth. 

To a dominant integral weight $\lambda=\sum_{k=1}^nm_k\omega_k\in\Lambda^+$, we assign a partition $\lambda = (\lambda_1\geq \lambda_2\geq \ldots \geq \lambda_n \geq 0)$ in the usual way, that is, by setting $\lambda_i= m_i+\ldots+m_n$. To such a partition, one attaches a Young diagram (we make use of the English convention), denoted by $Y_{\lambda}$.
A symplectic PBW tableau, $\tableau_{\lambda}$ of shape $\lambda$ is a filling of the corresponding Young diagram $Y_{\lambda}$ with numbers $\tableau_{i,j}\in \{1,\ldots,n,\overline{n},\ldots, \overline{1}\}$ such that for $\mu_j$, the length of the $j$-th column, we have:

 \begin{enumerate}
  \item[(i)] if $\tableau_{i,j}\leq \mu_j$, then $\tableau_{i,j}=i$;
  \item[(ii)] if $\tableau_{i_1,j}\neq i_1$, and $i_1<i_2$, then $\tableau_{i_1,j}>\tableau_{i_2,j}$;
  \item[(iii)] if $\tableau_{i,j}= i$, and $\exists \,\, i'$ such that $\tableau_{i',j}=\overline{i}$, then $i'<i$.
 \end{enumerate}
 
A symplectic PBW tableau is said to be PBW-semistandard if in addition, the following condition is satisfied: 

\begin{enumerate}
 \item[(iv)] for every $j>1$ and every $i$, $\exists \,\, i'\geq i$ such that $\tableau_{i',j-1}\geq \tableau_{i,j}$.
\end{enumerate}

Let $\SyST_{\lambda}$ denote the set of all symplectic PBW-semistandard tableaux of shape $\lambda$. To each $\tableau\in \SyST_{\lambda}$, we associate the monomial
$\X_{\tableau}=\prod_{j=1}^{\lambda_1}\X_{\tableau_{1,j},\ldots,\tableau_{\mu_j,j}}\in \V_{\lambda}^*.$

\begin{theorem}[\cite{B20}]\label{Thm:SYST}
The elements $\X_{\tableau}$, $\tableau\in \SyST_{\lambda}$, form a basis of $\mathbb{C}[\complete_{2n}]_{\lambda}$.
\end{theorem}

We will need the following construction in Section \ref{Sec:MaxCone}.

For $\mathrm{J}=(j_1,\ldots,j_d)\in\mathcal{P}$ with $1\leq j_1<\ldots<j_k\leq d$. We consider the strip tableau $T:=T(J)$ obtained as follows: first fill the $j_1,\ldots,j_k$-th box from top by $j_1,\ldots,j_k$; then fill the rest of the diagram from bottom to top by $j_{k+1},\ldots,j_d$. Such a tableau $T$ is not necessarily a symplectic PBW tableau: from construction the conditions (i) and (ii) are fulfilled, but the condition (iii) does not always hold. We will construct another tableau $\mathrm{st}(T)$ called the \emph{standardization} of $T$.

Assume that the condition (iii) is violated in $T$, then there exists $1\leq i\leq d$ such that $\ov{i}$ appears at the $p$-th box from top (notice that by (i), $i$ appears at the $i$-th box from top). The assumption ensures $i<p$. We define a new tableau $T_1$ of the same shape as $T$, which differs to $T$ at the $i$-th and the $p$-th boxes, where the $i$-th (resp. $p$-th) box of $T_1$ is filled by $\ov{p}$ (resp. $p$). From construction, this tableau $T_1$ satisfies the conditions (i) and (ii), and the number of $i$, such that $1\leq i\leq d$ and the condition (iii) is violated at $i$, decreases by one. Repeating the above procedure to $T_1$ until the condition (iii) is violated nowhere, the result is a symplectic PBW tableau $\mathrm{st}(T)$.

\subsection{Birational sequences and affine chart}\label{Section:BirSeq}

We introduce a chart of $\complete_{2n}^{\mathbf{d}}$ motivated by birational sequences introduced in \cite{FFL17}.

We enumerate the positive roots in $\Phi^+=\{\beta_1,\beta_2,\ldots,\beta_N\}$ in such a way that if $\beta_i-\beta_j$ is a sum of positive roots, then $i<j$. Such a sequence of positive roots $(\beta_1,\beta_2,\ldots,\beta_N)$ is called a \emph{good sequence} in \emph{loc.cit.} This fixed enumeration gives an isomorphism of affine varieties
$$\mathbb{C}^N\stackrel{\sim}{\longrightarrow}\mathrm{N},\ \ (t_1,\ldots,t_N)\mapsto \exp(t_1f_{\beta_1})\cdots\exp(t_Nf_{\beta_N}).$$
In view of the Pl\"ucker embedding
$$\complete_{2n}=\overline{\mathrm{N}\cdot ([\nu_{\omega_1}],\ldots,[\nu_{\omega_n}])}\hookrightarrow\mathbb{P}\Big({\bigwedge}^1\mathbb{C}^{2n} \Big)\times \ldots\times \mathbb{P}\Big({\bigwedge}^n\mathbb{C}^{2n} \Big),$$
for a fixed $x=\exp(t_1f_{\beta_1})\cdots\exp(t_Nf_{\beta_N})\in \mathrm{N}$ and $J\in\mathcal{P}$ with $k:=|J|$, the value $X_J(x\cdot ([\nu_{\omega_1}],\ldots,[\nu_{\omega_n}]))=X_J(x\cdot [\nu_{\omega_k}])$. When $t_1,\ldots,t_N$ vary in $\mathbb{C}$, we obtain a polynomial 
$$p_J:=X_J(x\cdot [\nu_{\omega_k}])\in\mathbb{C}[t_1,\ldots,t_N].$$

We define an algebra morphism
$$\varphi:\mathcal{S}\to\mathbb{C}[t_1,\ldots,t_N,z_1,\ldots,z_n],\ \ X_J\mapsto p_Jz_{|J|},$$
where the variables $z_1,\ldots,z_n$ are used to reflect the multi-homogeneity of the Pl\"ucker embedding. We consider the $\Lambda^+$-grading on $\mathbb{C}[t_1,\ldots,t_N,z_1,\ldots,z_n]$ by assigning degree $\omega_k$ to $z_k$, $1\leq k\leq n$ and degree $0$ to $t_i$, $1\leq i\leq N$. With the natural $\Lambda^+$-grading on $\mathcal{S}$, $\varphi$ is $\Lambda^+$-graded. The following lemma follows from definition.

\begin{lemma}\label{Lem:PolMap1}
We have: $\ker\varphi=\mathfrak{I}_{2n}$.
\end{lemma}

In the degenerate setting, we consider the following isomorphism of affine varieties
$$\mathbb{C}^N\stackrel{\sim}{\longrightarrow}\mathrm{N}^{\mathbf{d}},\ \ (t_1,\ldots,t_N)\mapsto \exp(t_1f_{\beta_1}^{\mathbf{d}})\cdots\exp(t_Nf_{\beta_N}^{\mathbf{d}}).$$
We have fixed the Pl\"ucker embedding
$$\complete_{2n}^{\mathbf{d}}=\overline{\mathrm{N}^{\mathbf{d}}\cdot ([\nu_{\omega_1}^{\mathbf{d}}],\ldots,[\nu_{\omega_n}^{\mathbf{d}}])}\hookrightarrow\mathbb{P}\Big({\bigwedge}^1\mathbb{C}^{2n} \Big)\times \ldots\times \mathbb{P}\Big({\bigwedge}^n\mathbb{C}^{2n} \Big).$$
We define a grading on $\mathbb{C}[t_1,\ldots,t_N]$ by assigning degree $-d_{\beta_i}$ to the variable $t_i$. Then for $x=\exp(t_1f_{\beta_1}^{\mathbf{d}})\cdots\exp(t_Nf_{\beta_N}^{\mathbf{d}})\in\mathrm{N}^{\mathbf{d}}$, $X_J^{\mathbf{d}}(x\cdot ([\nu_{\omega_1}^{\mathbf{d}}],\ldots,[\nu_{\omega_n}^{\mathbf{d}}]))=\mathrm{in}_{\mathbf{d}}(p_J)\in\mathbb{C}[t_1,\ldots,t_N]$, where $\mathrm{in}_{\mathbf{d}}$ is the initial term with respect to the above grading on $\mathbb{C}[t_1,\ldots,t_N]$. We set $p_J^{\mathbf{d}}:=\mathrm{in}_{\mathbf{d}}(p_J)$. This notation $\mathrm{in}_{\mathbf{d}}$ can be extended to $\mathbb{C}[t_1,\ldots,t_N,z_1,\ldots,z_n]$ by requiring the variables $z_1,\ldots,z_n$ to have degree $0$. We will denote this degree of a polynomial $p\in\mathbb{C}[t_1,\ldots,t_N,z_1,\ldots,z_n]$ by $\mathrm{deg}_{\mathbf{d}}(p)$.

Similarly we define an algebra morphism
$$\varphi^{\mathbf{d}}:\mathcal{S}^{\mathbf{d}}\to\mathbb{C}[t_1,\ldots,t_N,z_1,\ldots,z_n],\ \ X_J^{\mathbf{d}}\mapsto p_J^{\mathbf{d}}z_{|J|}.$$
With the same $\Lambda^+$-grading as above, $\varphi^{\mathbf{d}}$ is $\Lambda^+$-graded. 

\begin{lemma}\label{Lem:PolMap2}
We have: $\ker\varphi^{\mathbf{d}}=\mathfrak{I}_{2n}^{\mathbf{d}}$.
\end{lemma}

\begin{proof}
This lemma follows from the following rephrasing of the definition: for a $\Lambda^+$-homogeneous element $f\in\mathcal{S}^{\mathbf{d}}$ of degree $\mu=\mu_1\omega_1+\ldots+\mu_n\omega_n$, then 
$$\varphi^{\mathbf{d}}(f)=f(x\cdot ([\nu_{\omega_1}^{\mathbf{d}}],\ldots,[\nu_{\omega_n}^{\mathbf{d}}]))z_1^{\mu_1}\cdots z_n^{\mu_n}$$
where $x=\exp(t_1f_{\beta_1}^{\mathbf{d}})\cdots\exp(t_Nf_{\beta_N}^{\mathbf{d}})\in\mathrm{N}^{\mathbf{d}}$ is a generic element.
\end{proof}

\begin{example}
We illustrate the construction above in an example. For $\mathfrak{sp}_4$, we choose the enumeration $(\beta_1,\beta_2,\beta_3,\beta_4)=(\alpha_{1,\ov{1}},\alpha_{1,2},\alpha_{1,1},\alpha_{2,2})$. 

For a fixed $x=\exp(t_1f_{1,\ov{1}})\exp(t_2f_{1,2})\exp(t_3f_{1,1})\exp(t_4f_{2,2})$, we have
$$x\cdot e_1=e_1+t_3e_2+t_2e_{\ov{2}}+(t_1+t_2t_3)e_{\ov{1}}$$
and
$$x\cdot e_1\wedge e_2=e_1\wedge e_2+t_4e_1\wedge e_3+(t_2-t_3t_4)e_1\wedge e_{\ov{1}}+(t_3t_4-t_2)e_2\wedge e_{\ov{2}}+$$
$$+(-t_1-t_3^2t_4)e_2\wedge e_{\ov{1}}+(t_2^2-t_1t_4-2t_2t_3t_4)e_{\ov{2}}\wedge e_{\ov{1}}.$$
The map $\varphi$ sends
$$X_1\mapsto z_1,\ \ X_2\mapsto t_3z_1,\ \ X_{\ov{2}}\mapsto t_2z_1,\ \ X_{\ov{1}}\mapsto (t_1+t_2t_3)z_1,\ \ X_{12}\mapsto z_2,\ \ X_{1{\ov{2}}}\mapsto t_4z_2,$$
$$X_{1{\ov{1}}}\mapsto (t_2-t_3t_4)z_2,\ \ X_{2{\ov{2}}}\mapsto (t_3t_4-t_2)z_2,\ \ X_{2{\ov{1}}}\mapsto -(t_1+t_3^2t_4)z_2,\ \ X_{{\ov{2}}{\ov{1}}}\mapsto (t_2^2-t_1t_4-2t_2t_3t_4)z_2.$$
It is straightforward to verify that the defining relations introduced in Section \ref{Sec:DefRel} are in the kernel of $\varphi$.

We choose $\mathbf{d}\in\mathcal{K}_{4}$ with $d_{1,1}=3$, $d_{2,2}=1$, $d_{1,2}=2$ and $d_{1,\ov{1}}=1$. Such a point $\mathbf{d}$ is in the interior of $\mathcal{K}_4$. The variables $t_1$, $t_2$, $t_3$ and $t_4$ get degree $-1$, $-2$, $-3$ and $-1$ respectively. The map $\varphi^{\mathbf{d}}$ sends
$$X_1^{\mathbf{d}}\mapsto z_1,\ \ X_2^{\mathbf{d}}\mapsto t_3z_1,\ \ X_{\ov{2}}^{\mathbf{d}}\mapsto t_2z_1,\ \ X_{\ov{1}}^{\mathbf{d}}\mapsto t_1z_1,\ \ X_{12}^{\mathbf{d}}\mapsto z_2,\ \ X_{1{\ov{2}}}^{\mathbf{d}}\mapsto t_4z_2,$$
$$X_{1{\ov{1}}}^{\mathbf{d}}\mapsto t_2z_2,\ \ X_{2{\ov{2}}}^{\mathbf{d}}\mapsto -t_2z_2,\ \ X_{2{\ov{1}}}^{\mathbf{d}}\mapsto -t_1z_2,\ \ X_{{\ov{2}}{\ov{1}}}^{\mathbf{d}}\mapsto -t_1t_4z_2.$$
The following relations are contained in $\ker\varphi^{\mathbf{d}}$:
$$\X_{12}\X_{\ov{2}}+\X_{2\ov{2}}\X_1,\ \ \X_{1\ov{2}}\X_{\ov{1}}+\X_{\ov{2}\ov{1}}\X_1,\ \ \X_{2\ov{2}}\X_{\ov{1}}-\X_{2\ov{1}}\X_{\ov{2}},$$
$$\X_{12}\X_{\ov{1}}+\X_{2\ov{1}}\X_1,\ \ \X_{12}\X_{\ov{2}\ov{1}}-\X_{1\ov{2}}\X_{2\ov{1}},\ \ \X_{1\ov{1}}+\X_{2\ov{2}}.$$
\end{example}

We define a monomial order $>_r$ on $\mathbb{C}[t_1,\ldots,t_N,z_1,\ldots,z_n]$ by requiring: two monomials 
$$t_1^{a_1}\cdots t_N^{a_N}z_1^{\lambda_1}\cdots z_n^{\lambda_n}>_r t_1^{b_1}\cdots t_N^{b_N}z_1^{\mu_1}\cdots z_n^{\mu_n}$$
if  the last non-zero coordinate of 
$$(a_1,\ldots,a_N,\lambda_1,\ldots,\lambda_n)-(b_1,\ldots,b_N,\mu_1,\ldots,\mu_n)$$ 
is positive. We define a valuation 
$$\nu_{>_r}:\mathbb{C}[t_1,\ldots,t_N,z_1,\ldots,z_n]\to\mathbb{N}^N\times\mathbb{N}^n,$$
sending a polynomial $f$ to the minimal exponent appearing in $f$ with respect to $>_r$. Let $p_1:\mathbb{N}^N\times\mathbb{N}^n\to\mathbb{N}^N$ and $p_2:\mathbb{N}^N\times\mathbb{N}^n\to\mathbb{N}^n$ be the projections to the corresponding components.

Recall that a symplectic PBW-semistandard tableau $T\in\SyST_{\omega_k}$ is a strip of length $k$ filled by elements in the set $\{1,\ldots,n,\overline{n},\ldots,\overline{1}\}$. Such a tableau gives an element $J(T)\in\mathcal{P}$ as the ordered set of numbers appearing in $T$.

For $1\leq k\leq n$, we define a map 
$$\rho_k:\SyST_{\omega_k}\to\set(\omega_k)\subseteq\mathbb{N}^{\Phi^+},\ \ T\mapsto p_1(\nu_{>_r}(\varphi(X_{J(T)})))\in\mathbb{N}^N,$$
here we identify $\mathbb{N}^N$ with $\mathbb{N}^{\Phi^+}$ by sending the coordinate $e_i$ to the coordinate function $e_{\beta_i}$.

\begin{lemma}\label{Lem:TabBij}
The map $\rho_k$ is a bijection.
\end{lemma}

\begin{proof}
According to Theorem \ref{thm:FFLV-degenerate} and Theorem \ref{Thm:SYST}, both sets have the same cardinality $\dim\V_{\omega_k}$. To show the surjectivity, we construct a map from $\set(\omega_k)$ to $\SyST_{\omega_k}$. Take $\mathbf{s}\in\set(\omega_k)$ with support $\{\alpha_{i_1,j_1},\ldots,\alpha_{i_r,j_r}\}$, from the proof of Lemma \ref{Lem:StdDec}, such a support is an anti-chain in $\Phi_k^+$. As consequences we have for $1\leq s\leq r$, $i_s\leq k\leq j_s$, and if we assume that $i_1<\ldots<i_r$, then $j_1>\ldots>j_r$.

We define a strip tableau $T$ of length $k$ as follows: start with the tableau filled with $1,\ldots,k$ from top to bottom, then for $1\leq s\leq r$, replace $i_s$ by either $j_s+1$ if $j_s<n$, or $j_s$ otherwise. It remains to show that $T$ is a symplectic PBW tableau. Indeed, the condition (i) and (iii) are fulfilled from construction, and the property $j_1>\ldots>j_r$ gives the condition (ii). 

It remains to show that $\rho_k(T)=\mathbf{s}$. First notice that by definition of $\varphi$, all monomials appearing in $\varphi(X_{J(T)})$ are of form $t_1^{r_1}\cdots t_N^{r_N}z_k$, and the positive integers $r_1,\ldots,r_N$ satisfies: 
\begin{equation}\label{Eq:RootPar}
r_1\beta_1+\ldots+r_N\beta_N=\omega_k-\mathrm{wt}(e_{J(T)}),
\end{equation}
where $\mathrm{wt}(e_{J(T)})$ is the weight of the element $e_{J(T)}\in\V_{\omega_k}$. Consider the set of all possible tuples $(r_1,\ldots,r_N)\in\mathbb{N}^N$ satisfying \eqref{Eq:RootPar}, the monomial order $>_r$ induces a total order on these tuples. We claim that the minimal element is $\mathbf{s}_{[k],J(T)}$, looked in $\mathbb{N}^N$. Indeed, from the definition of the monomial order, we would prefer to split $\omega_k-\mathrm{wt}(e_{J(T)})$ into positive roots of higher heights. This gives a greedy procedure: first check whether $\omega_k-\mathrm{wt}(e_{J(T)})-\beta_1$ is in the monoid $\mathbb{N}\Phi^+$ generated by positive roots, if yes then proceed by considering this new weight $\omega_k-\mathrm{wt}(e_{J(T)})-\beta_1$, otherwise move to $\beta_2$ and repeat the procedure with the original weight $\omega_k-\mathrm{wt}(e_{J(T)})$. The output of this algorithm is the element $\mathbf{s}_{[k],J(T)}$. Since $\mathbf{s}_{[k],J(T)}\in M_{[k]}^{J(T)}$ (see Section \ref{Sec:GenCase} for this notation), we proved that $\rho_k(T)=\mathbf{s}_{[k],J(T)}$ when looked in $\mathbb{N}^N$.

To conclude it suffices to notice that $\mathbf{s}=\mathbf{s}_{[k],J(T)}$ (see \eqref{Eq:SIJ}).
\end{proof}

For an arbitrary $\lambda\in\Lambda^+$ and $T\in\SyST_\lambda$, we let $T^1,\ldots,T^m$ be the columns of $T$ with $T^k\in\SyST_{\omega_{i_k}}$. The maps $\rho_1,\ldots,\rho_n$ can be merged together to give a map 
$$\rho_\lambda:\SyST_\lambda\to\set(\lambda),\ \ T\mapsto \rho_{i_1}(T^1)+\ldots+\rho_{i_m}(T^m).$$
By Lemma \ref{Lem:TabBij} and Theorem \ref{thm:FFLV-degenerate} (1), the map $\rho_\lambda$ is well-defined.

\begin{proposition}\label{Prop:TabFFLV}
The map $\rho_\lambda$ is a bijection.
\end{proposition}

\begin{proof}
Again by Theorem \ref{thm:FFLV-degenerate} and Theorem \ref{Thm:SYST}, it suffices to show the surjectivity. Given $\mathbf{s}\in\set(\lambda)$, let $\mathbf{s}=\mathbf{s}_1+\ldots+\mathbf{s}_m$ be the standard decomposition of $\mathbf{s}$ in Section \ref{Sec:SymplecticFFLV} with $\mathbf{s}_k\in \set(\omega_{i_k})$. We define a tableau $T$ by stacking the strips $T^1:=\rho_{i_1}^{-1}(\mathbf{s}_1),\ldots,T^m:=\rho_{i_m}^{-1}(\mathbf{s}_m)$ from left to right. By Lemma \ref{Lem:TabBij}, $T$ is a symplectic PBW tableau.

We show that $T$ is PBW-semistandard. Let $T^r_j$ denote the $j$-th element in $T^r$ from the top. The condition (iv) involves only two neighbored columns, we look at the strip tableaux $T^\ell$ and $T^{\ell+1}$. If  $T^{\ell+1}_j=j$, the condition (iv) is fulfilled at this place by the condition (ii) for $T^{\ell}$. Otherwise assume that $T_j^{\ell+1}=p$ with $p\neq j$. If $p\leq n$, then $\alpha_{j,p-1}$ is contained in the support of $\mathbf{s}_{\ell+1}$. From the construction of the standard decomposition, in the support of $\mathbf{s}_\ell$ there must be a positive root $\alpha_{q,t}$  with $q\leq k\leq t$, $j\leq q$ and $p-q\leq t$. From the construction in Lemma \ref{Lem:TabBij}, $T^{\ell}_p\geq p$. If $p\geq\overline{n}$, $\alpha_{j,p}$ is contained in the support of $\mathbf{s}_{\ell+1}$, with the same argument one verifies the condition (iv).
\end{proof}

\subsection{Proof of Theorem ~\ref{Thm:main2}}

We are ready to prove Theorem \ref{Thm:main2}. The inclusion $\mathrm{in}_{\mathbf{w}}(\mathfrak{I}_{2n})\subseteq\mathfrak{I}_{2n}^{\mathbf{d}}$ follows from the following claim: for any $f\in\mathcal{S}$, $\varphi^{\mathbf{d}}(\mathrm{in}_{\mathbf{w}}(f))=\mathrm{in}_{\mathbf{d}}(\varphi(f))$, where $\mathrm{in}_{\mathbf{w}}(f)$ is looked in $\mathcal{S}^{\mathbf{d}}$. Assume that the claim is established, we show that for any $f\in\mathfrak{I}_{2n}$, $\mathrm{in}_{\mathbf{w}}(f)\in\mathfrak{I}_{2n}^{\mathbf{d}}$. According to Lemma \ref{Lem:PolMap2}, it suffices to show that $\varphi^{\mathbf{d}}(\mathrm{in}_{\mathbf{w}}(f))=0$. By Lemma \ref{Lem:PolMap1}, $\varphi(f)=0$, applying the claim proves the inclusion. 

To show the claim, first notice that when $f=X_J$ for some $J\in\mathcal{P}$, $\mathrm{in}_{\mathbf{w}}(X_J^{\mathbf{d}})=X_J^{\mathbf{d}}$ and the identity $\varphi^{\mathbf{d}}(X_J^{\mathbf{d}})=\mathrm{in}_{\mathbf{d}}(\varphi(X_J))$ is the definition of $\varphi^{\mathbf{d}}$. If $f$ is a monomial in $X_J$, the identity holds since both $\varphi$ and $\varphi^{\mathbf{d}}$ are algebra morphisms, and $\mathrm{in}_{\mathbf{d}}$ preserves products of polynomials. For an arbitrary polynomial $f\in\mathcal{S}$, assume that $\mathrm{in}_{\mathbf{w}}(f)=c_1X^{\mathbf{a}_1}+\ldots+c_kX^{\mathbf{a}_k}$ where for $1\leq i\leq k$, $\mathbf{a}_i\in\mathbb{N}^{\mathcal{P}}$, $X^{\mathbf{a}_i}:=\prod_{J\in\mathcal{P}}X_J^{\mathbf{a}_i(J)}$ and $c_i\in\mathbb{C}\setminus\{0\}$. Moreover, the degree $d:=\mathrm{deg}_{\mathbf{w}}(X^{\mathbf{a}_i}):=\sum_{J\in\mathcal{P}}\mathbf{w}(J)\mathbf{a}_i(J)$ is the same for $1\leq i\leq k$. From the definition of $\mathbf{w}$ and the grading on $\mathbb{C}[t_1,\ldots,t_N,z_1,\ldots,z_n]$, we have: for $1\leq i\leq k$, $\mathrm{deg}_{\mathbf{w}}(X^{\mathbf{a}_i})=\mathrm{deg}_{\mathbf{d}}(\varphi(X^{\mathbf{a}_i}))$. For a monomial $X^{\mathbf{b}}$ in $f$ which does not appear in $\mathrm{in}_{\mathbf{w}}(f)$, all monomials in $\varphi(X^{\mathbf{b}})$ will have strictly smaller degree than $d$, hence they do not contribute to $\mathrm{in}_{\mathbf{d}}(\varphi(f))$. Moreover, since $\mathrm{deg}_{\mathbf{d}}(\varphi(X^{\mathbf{a}_i}))$ are the same for $1\leq i\leq k$, it then follows
$$\mathrm{in}_{\mathbf{d}}(\varphi(c_1X^{\mathbf{a}_1}+\ldots+c_kX^{\mathbf{a}_k}))=\mathrm{in}_{\mathbf{d}}(\varphi(c_1X^{\mathbf{a}_1}))+\ldots+\mathrm{in}_{\mathbf{d}}(\varphi(c_kX^{\mathbf{a}_k})).$$
As a summary,
\begin{eqnarray*}
\varphi^{\mathbf{d}}(\mathrm{in}_{\mathbf{w}}(f))&=& \varphi^{\mathbf{d}}(c_1X^{\mathbf{a}_1}+\ldots+c_kX^{\mathbf{a}_k})\\
&=& \varphi^{\mathbf{d}}(c_1X^{\mathbf{a}_1})+\ldots+\varphi^{\mathbf{d}}(c_kX^{\mathbf{a}_k})\\
&=& \mathrm{in}_{\mathbf{d}}(\varphi(c_1X^{\mathbf{a}_1}))+\ldots+ \mathrm{in}_{\mathbf{d}}(\varphi(c_kX^{\mathbf{a}_k}))\\
&=& \mathrm{in}_{\mathbf{d}}(\varphi(c_1X^{\mathbf{a}_1}+\ldots+c_kX^{\mathbf{a}_k}))=\mathrm{in}_{\mathbf{d}}(\varphi(f)).
\end{eqnarray*}
The proof of the claim is then complete.

To show the other inclusion, first notice that since $\varphi^{\mathbf{d}}$ is $\Lambda^+$-graded, so is $\mathcal{S}^{\mathbf{d}}/\mathfrak{I}_{2n}^{\mathbf{d}}$; moreover, the ideal $\mathrm{in}_{\mathbf{w}}(\mathfrak{I}_{2n})$ is an initial ideal of $\mathfrak{I}_{2n}$ which is $\Lambda^+$-graded, hence the algebra $\mathcal{S}^{\mathbf{d}}/\mathrm{in}_{\mathbf{w}}(\mathfrak{I}_{2n})$ is also $\Lambda^+$-graded. We prove the other inclusion by comparing the dimension of the component of degree $\lambda\in\Lambda^+$ in both $\mathcal{S}^{\mathbf{d}}/\mathfrak{I}_{2n}^{\mathbf{d}}$ and $\mathcal{S}^{\mathbf{d}}/\mathrm{in}_{\mathbf{w}}(\mathfrak{I}_{2n})$. Being an initial ideal of $\mathfrak{I}_{2n}$, the dimension of the degree $\lambda$ component of $\mathcal{S}^{\mathbf{d}}/\mathrm{in}_{\mathbf{w}}(\mathfrak{I}_{2n})$ coincides with that of $\mathcal{S}/\mathfrak{I}_{2n}$, which is $\dim\V_\lambda$. We show that the degree $\lambda$ component of $\mathcal{S}^{\mathbf{d}}/\mathfrak{I}_{2n}^{\mathbf{d}}$ has the same dimension. 

For this we consider the monomial $X^{\mathbf{d}}_T:=X^{\mathbf{d}}_{J(T^1)}\cdots X^{\mathbf{d}}_{J(T^m)}$ where $T^1,\ldots,T^m$ are the columns of a symplectic PBW-semistandard tableau $T\in\SyST_\lambda$. We claim that the elements $\{\varphi^{\mathbf{d}}(X^{\mathbf{d}}_T)\mid T\in\SyST_\lambda\}$ are linearly independent. This will terminate the proof as by Theorem \ref{Thm:SYST}, the cardinality of $\SyST_\lambda$ is $\dim\V_\lambda$.

To prove the claim, we observe that for any $J\in\mathcal{P}$, $\nu_{>_r}(\varphi(X_J))=\nu_{>_r}(\varphi^{\mathbf{d}}(X_J^{\mathbf{d}}))$. Indeed, assume that $|J|=k$, and denote by $(s_1,\ldots,s_N)$ the element $\mathbf{s}_{[k],J}$ under the identification between $\mathbb{N}^{\Phi^+}$ and $\mathbb{N}^N$.  From the proof of Lemma \ref{Lem:TabBij}, we have shown that the minimal term in $p_J$ with respect to $>_r$ is $t_1^{s_1}\cdots t_N^{s_N}z_k$. By the argument in the beginning of Section \ref{Sec:Boundary}, this monomial appears in $p_J^{\mathbf{d}}$ hence it is also the minimal term in $p_J^{\mathbf{d}}$ with respect to $>_r$. From this observation, 
\begin{eqnarray*}
\nu_{>_r}(\varphi^{\mathbf{d}}(X_T^{\mathbf{d}})) &=& \nu_{>_r}(\varphi^{\mathbf{d}}(X_{J(T^1)}^{\mathbf{d}}))+\ldots+\nu_{>_r}(\varphi^{\mathbf{d}}(X_{J(T^m)}^{\mathbf{d}}))\\
&=& \nu_{>_r}(\varphi(X_{J(T^1)}))+\ldots+\nu_{>_r}(\varphi(X_{J(T^m)}))\\
&=& (\rho_\lambda(T),\lambda_1,\ldots,\lambda_n)
\end{eqnarray*}
where $\rho_\lambda(T)\subseteq\mathbb{N}^{\Phi^+}$ is looked in $\mathbb{N}^N$. When $T$ runs over $\SyST_\lambda$, by Proposition \ref{Prop:TabFFLV}, $\rho_\lambda(T)$ are pairwise different, this proves the claim, and the proof of Theorem \ref{Thm:main2} (1) is thus complete.

In the last part of the proof we have shown
\begin{corollary}\label{Cor:Basis}
The set $\{X_T^{\mathbf{d}}\mid T\in\SyST_\lambda\}$ form a basis of the multi-homogeneous coordinate ring $\mathcal{S}^{\mathbf{d}}/\mathfrak{I}_{2n}^{\mathbf{d}}$ of $\complete_{2n}^{\mathbf{d}}$.
\end{corollary}

\subsection{Proof of Theorem \ref{Thm:main3} (2)}\label{Sec:ProofThm(2)}

We consider the polynomial map $\varphi^{\mathbf{d}}$. When $\mathbf{d}$ is chosen from the interior of $\mathcal{K}_{2n}$, from Lemma \ref{Lem:Fundk} and Lemma \ref{Lem:TabBij} we know that for any $T\in\SyST_{\omega_k}$, $\mathrm{in}_{\mathbf{d}}(p_{J(T)})$ is the monomial $t_1^{s_1}\cdots t_N^{s_N}z_{k}$ where $(s_1,\ldots,s_N)$ is $\mathbf{s}_{[k],J(T)}$ under the identification of $\mathbb{N}^{\Phi^+}$ to $\mathbb{N}^N$. The map $\SyST_{\omega_k}\to\set(\omega_k)$, $T\mapsto z_k^{-1}\mathrm{in}_{\mathbf{d}}(p_{J(T)})$ is hence a bijection. By Theorem \ref{thm:FFLV-degenerate}, the degree $\lambda$ component of $\mathcal{S}^{\mathbf{d}}/\mathfrak{I}_{2n}^{\mathbf{d}}$ can be identified by $\varphi^{\mathbf{d}}$ with monomials having lattice points in $\mathrm{FFLV}(\lambda)$ as exponents. 

When we only look the homogeneous components of having multiples of $\lambda$ as degrees, they form the homogeneous coordinate ring of $\complete_{2n}^{\mathbf{d}}$ in $\mathbb{P}(\V_\lambda^{\mathbf{d}})$. From the above argument, such a ring is the homogeneous coordinate ring of the toric variety associated to the polytope $\mathrm{FFLV}(\lambda)$. The proof is thus complete.

\subsection{An analogue of the Borel-Weil theorem}

Recall the embedding  
$$\iota_{\lambda}: \complete_{2n}^{\mathbf{d}}\hookrightarrow \mathbb{P}(\V_{\lambda}^{\mathbf{d}})$$
from Subsection \ref{Sec:WeightedDeg} above. Consider the pull-back of the canonical line bundle $\mathcal{O}_{\mathbb{P}(\V_{\lambda}^{\mathbf{d}})}(1)$ with respect to the embedding: 
$$\mathcal{L}_{\lambda}^{\mathbf{d}}:= \iota_{\lambda}^*(\mathcal{O}_{\mathbb{P}(\V_{\lambda}^{\mathbf{d}})}(1)).$$ 
Recall the degenerate group $\N^{\mathbf{d}}$ acting on $\V_{\lambda}^{\mathbf{d}}$. By an argument similar to the proof of \cite[Theorem 8.1]{FFFM19} and from Corollary \ref{Cor:Basis}, we deduce the following analogue of the Borel-Weil theorem in our context.

 \begin{theorem}\label{Thm:Borel1}
For every $\mathbf{d}\in\mathcal{K}_{2n}$, we have $\mathrm{H}^0(\complete_{2n}^{\mathbf{d}},\mathcal{L}_{\lambda}^{\mathbf{d}})\simeq (\V_{\lambda}^{\mathbf{d}})^*$ as $\N^{\mathbf{d}}$-modules.
 \end{theorem}

Moreover, the following Borel-Weil type vanishing theorem holds true.

\begin{theorem}\label{Thm:Borel2}
$\mathrm{H}^m(\complete_{2n}^{\mathbf{d}},\mathcal{L}_{\lambda}^{\mathbf{d}})= 0$ for all $\lambda\in \Lambda^+$ and all $m>0$.
\end{theorem}
\begin{proof}
    By \cite[Sec. 3.5]{Ful93}, the theorem holds for any $\mathbf{d}$ in the relative interior of $\mathcal{K}_{2n}$, since the line bundle $\mathcal{L}_{\lambda}^{\mathbf{d}}$ is generated by its sections. By Theorem \ref{Thm:main3} (2), it follows that for any generic point $\mathbf{d}'\in \mathcal{K}_{2n}$, one has a flat family over $\mathbb{A}^1$ with the generic fiber $\complete_{2n}^{\mathbf{d}'}$ and $\complete_{2n}^{\mathbf{d}}$ as special fiber. The claim of the theorem then follows from \cite[Theorem 12.8]{Har13}.
\end{proof}

\section{A maximal cone in $\mathrm{Trop}(\complete_{2n})$}\label{Sec:MaxCone}

In this section we describe the image of $w:\mathcal{K}_{2n}\to\mathbb{R}^{\mathcal{P}}$ defined in Section \ref{Sec:WeightedDeg}. 

First we extend the map $w$ to the entire $\mathbb{R}^{\Phi^+}$ using the same definition by choosing $\mathbf{d}$ from $\mathbb{R}^{\Phi^+}$. The extended map is linear and injective, and will be again denoted by $w:\mathbb{R}^{\Phi^+}\to\mathbb{R}^{\mathcal{P}}$. We shall denote the image $w(\mathcal{K}_{2n})$ by $\mathcal{C}_{2n}$, and set $s_J\in\mathbb{R}^{\mathcal{P}}$ to be the coordinate function corresponding to $J\in\mathcal{P}$. By the linearity of the map $w$, $\mathcal{C}_{2n}$ is a polyhedral cone, which is termed \emph{symplectic FFLV tropical cone}. We start from describing its defining inequalities.

We will use notations introduced in Section \ref{Sec:PBWTab} and Section \ref{Section:BirSeq}.

\begin{lemma}\label{lem:maximalcone}
The following set of equalities and inequalities cuts out the cone $\mathcal{C}_{2n}$ in $\mathbb{R}^{\mathcal{P}}$:
\begin{enumerate}
    \item[(i).] $s_{1,\ldots,k}=0$ for $1\leq k\leq n$;
    \item[(ii).] for any $i< j$ with $1\leq i \leq n$,  $1\leq j \leq \ov{1}$ and any $1 \leq k< \ell\leq n$ such that $i\leq k<\ell<j$, we have $s_{1,\ldots,i-1,i+1,\ldots,k,j}=s_{1,\ldots,i-1,i+1,\ldots,\ell,j}$;
    \item[(iii).] for any $J=(j_1,\ldots,j_k)\in\mathcal{P}$ such that $\mathrm{st}(T(J))=T(J)$, let $\{\alpha_{p_1,q_1},\ldots,\alpha_{p_\ell,q_\ell}\}$ be the support of $\rho_k(T(J))$ with $p_1<\cdots<p_\ell$ and $q_1>\cdots>q_\ell$; then $s_{j_1,\ldots,j_k} = s_{1,\ldots,p_1-1,q_1}+\cdots+s_{1,\ldots, p_\ell-1,q_\ell}$;
    \item[(iv).] for any $J\in\mathcal{P}$ such that $\mathrm{st}(T(J))\neq T(J)$, then one has $s_J=s_{\mathrm{st}(T(J))}$;
    \item[(v).] $s_{1,\ldots,i-1,i+1}+s_{1,\ldots,i,i+2}\leq s_{1,\ldots,i-1,i+2}$ for $1\leq i\leq n-1$;
    \item[(vi).] $s_{1,\ldots,i-1,j}+s_{1,\ldots,i,j+1}\leq s_{1,\ldots,i-1,j+1}+s_{1,\ldots,i,j}$ for $1\leq i<j\leq n-1$;
    \item[(vii).] $s_{1,\ldots,i-1,\ov{j+1}}+s_{1,\ldots,i,\ov{j}}\leq s_{1,\ldots,i-1,\ov{j}}+s_{1,\ldots,i,\ov{j+1}}$ for $1\leq i<j\leq n-1$;
    \item[(viii).] $2s_{1,\ldots,i-1,\ov{i+1}}\leq s_{1,\ldots,i-1,\ov{i}}+s_{1,\ldots,i,\ov{i+1}}$ for $1\leq i\leq n-1$.
\end{enumerate}
\end{lemma}

\begin{proof}
The linear equalities (i), (iii) follow directly from the definition of $w$. Condition (ii) follows from the fact that for any $i<k<j$ with $k\leq n$, one has $s_{1,\ldots,i-1,i+1,\ldots,k,j}=a_{i,j}$. The equality (iv) holds by the construction of the standardization of a strip tableau in Section \ref{Sec:PBWTab}: for $J\in\mathcal{P}$ such that $\mathrm{st}(T(J))\neq T(J)$ with $\mathrm{st}(T(J))=T(J')$ for some $J'\in\mathcal{P}$, $\mathbf{s}_{[k],J}=\mathbf{s}_{[k],J'}$. The image of $w:\mathbb{R}^{\Phi^+}\to\mathbb{R}^{\mathcal{P}}$ is cut out in $\mathbb{R}^{\mathcal{P}}$ by the linear equalities (i)-(iv). The inequalities (v)-(viii) follow from Definition \ref{def:weightsystem}. By Proposition \ref{Prop:Cone} (2), the equalities and inequalities (i)-(viii) characterize the image of $\mathcal{K}_{2n}$ under $w$.
\end{proof}

From Proposition \ref{Prop:Cone} (1), $\mathcal{C}_{2n}$ is an $n^2$--dimensional cone in $\mathbb{R}^{\mathcal{P}}$.

\begin{example}
When $n=2$, the cone $\mathcal{C}_4\subseteq\mathbb{R}^{\mathcal{P}}$ is defined by the following equalities and inequalities:
$$s_1=s_{12}=0,\ \ s_{\ov{2}}=s_{1,\ov{1}} = s_{2,\ov{2}},\ \  s_{\ov{1}} = s_{2,\ov{1}},\ \ s_{\ov{2},\ov{1}} = s_{\ov{1}} +s_{1,\ov{2}},$$ 
$$s_{2} + s_{1,\ov{2}} \geq s_{\ov{2}},\ \ s_2 + s_{\ov{2}} \geq s_{\ov{1}},\ \ 2s_{\ov{2}}\geq s_{\ov{1}} + s_{1,\ov{2}}.$$
\end{example}

\subsection*{Proof of Theorem ~\ref{Thm:main3} (1)}
First notice that any Pl\"ucker coordinate $X_J$ for $J\in\mathcal{P}$ can not vanish identically on any $\complete_{2n}^{\mathbf{d}}$ for $\mathbf{d}\in\mathcal{K}_{2n}$. From Theorem \ref{Thm:main2}, for any point in $\mathcal{C}_{2n}$, the initial ideal of $\mathfrak{I}_{2n}$ associated to this point is a prime ideal, which contains no monomials, hence $\mathcal{C}_{2n}$ is contained in the tropical symplectic flag variety $\mathrm{Trop}(\complete_{2n})$. It remains to show that this cone is a maximal cone since the primeness follows from Theorem \ref{Thm:main2}. The idea of the proof is the same as \cite[Theorem 7.3]{FFFM19}. 

Since the dimension of $\mathrm{Trop}(\complete_{2n})$ is $n^2$, the cone $\mathcal{C}_{2n}$ is of maximal dimension. Assume there is a maximal cone $\mathcal{C}_{2n}'\subset\mathrm{Trop}(\complete_{2n})$ that contains $\mathcal{C}_{2n}$, then $\dim \mathcal{C}_{2n}' =n^2$ and $\mathcal{C}_{2n}'$ must be contained in the image of $w$. Therefore according to the proof of Lemma \ref{lem:maximalcone}, any point $s:=(s_J)_{J\in\mathcal{P}}\in \mathcal{C}_{2n}'$ must satisfy the linear equalities (i)-(iv) of Lemma \ref{lem:maximalcone}. We want to show that $\mathcal{C}_{2n}'\subseteq \mathcal{C}_{2n}$, so we are left with showing that the point $s$ satisfies the inequalities (v)-(viii) of Lemma \ref{lem:maximalcone}. We argue by showing that no point outside $\mathcal{C}_{2n}'$ lies in $\mathrm{Trop}(\complete_{2n})$. This is accomplished by considering each of the four types of facets corresponding to the inequalities (v)-(viii) as follows:
\begin{enumerate}
\item[(v)] We consider for $1\leq i\leq n-1$ the Pl\"ucker relation:
\[\X_{1,\ldots,i}\X_{1,\ldots,i-1,i+1,i+2}+\X_{1,\ldots,i-1,i+2}\X_{1,\ldots,i,i+1}-\X_{1,\ldots,i-1,i+1}\X_{1,\ldots,i,i+2}.\]
From Lemma \ref{lem:maximalcone} (i) and (ii), $s_{1,\ldots,i}=s_{1,\ldots,i,i+1}=0$ for $1\leq i\leq n-1$ , and $s_{1,\ldots,i-1,i+1,i+2}=s_{1,\ldots,i-1,i+2}$. If $s_{1,\ldots,i-1,i+2}<s_{1,\ldots,i-1,i+1}+s_{1,\ldots,i,i+2}$, then the initial form of the above Pl\"ucker relation is the monomial $-\X_{1,\ldots,i-1,i+1}\X_{1,\ldots,i,i+2}$. This means $s\notin\mathrm{Trop}(\complete_{2n})$, so the inequalities in (v) must be valid.

\item[(vi)] For $1\leq i< j\leq n-1$, we have the Pl\"ucker relation:
\[\X_{1,\ldots,i-1,j,j+1}\X_{1,\ldots,i,i+1}-\X_{1,\ldots,i-1,i+1,j+1}\X_{1,\ldots,i,j}+\X_{1,\ldots,i-1,i+1,j}\X_{1,\ldots,j+1}.\]
From (iii), we have the equality $s_{1,\ldots,i-1,j,j+1}=s_{1,\ldots,i-1,j+1}+s_{1,\ldots,i,j}$. Therefore, if $s_{1,\ldots,i-1,j}+s_{1,\ldots,i,j+1}>s_{1,\ldots,i-1,j+1}+s_{1,\ldots,i,j}$, the initial form of the above Pl\"ucker relation would be the monomial $\X_{1,\ldots,i-1,i+1,j}\X_{1,\ldots,j+1}$. This would imply again that $s\notin \mathrm{Trop}(\complete_{2n})$. It thus follows that the inequalities in (vi) must hold true.

\item[(vii)] Further still for $1\leq i<j\leq n-1$, we consider the Pl\"ucker relation: 
\[\X_{1,\ldots,i-1,i+1,\ov{j+1}}\X_{1,\ldots,i,\ov{j}}+\X_{1,\ldots,i-1,\ov{j+1},\ov{j}}\X_{1,\ldots,i,i+1}-\X_{1,\ldots,i-1,i+1,\ov{j}}\X_{1,\ldots,i,\ov{j+1}}.\]
We again have from (iii) the equality $s_{1,\ldots,i-1,\ov{j+1},\ov{j}}=s_{1,\ldots,i-1,\ov{j}}+s_{1,\ldots,i,\ov{j+1}}$. Therefore if $s_{1,\ldots,i-1,\ov{j+1}}+s_{1,\ldots,i,\ov{j}}> s_{1,\ldots,i-1,\ov{j}}+s_{1,\ldots,i,\ov{j+1}}$, then the initial form of the above Pl\"ucker relation is the monomial $\X_{1,\ldots,i-1,i+1,\ov{j+1}}\X_{1,\ldots,i,\ov{j}}$. This then implies that $s\notin \mathrm{Trop}(\complete_{2n})$. Inequalities in (vii) must therefore be valid.

\item[(viii)] Finally, for $1\leq i\leq n-1$ we consider the Pl\"ucker relation:
\[\X_{1,\ldots,i,\ov{i}}\X_{1,\ldots,i-1,i+1,\ov{i+1}}-\X_{1,\ldots,i-1,i+1,\ov{i}}\X_{1,\ldots,i,\ov{i+1}}+\X_{1,\ldots,i,i+1}\X_{1,\ldots,i-1,\ov{i+1},\ov{i}}.\]
By (iii) we have the equality $s_{1,\ldots,i-1,\ov{i+1},\ov{i}}=s_{1,\ldots,i-1,\ov{i}}+s_{1,\ldots,i,\ov{i+1}}$ and by (iv) we have the equality $s_{1,\ldots,i,\ov{i}}=s_{1,\ldots,i-1,i+1,\ov{i+1}}$. Therefore if $2s_{1,\ldots,i-1,\ov{i+1}}> s_{1,\ldots,i-1,\ov{i}}+s_{1,\ldots,i,\ov{i+1}}$, then the initial form of the above Pl\"ucker relation is the monomial $\X_{1,\ldots,i,\ov{i}}\X_{1,\ldots,i-1,i+1,\ov{i+1}}$. This means that $s\notin \mathrm{Trop}(\complete_{2n})$, hence inequalities in (viii) must hold true. 
\end{enumerate}

The above argument shows that $\mathcal{C}_{2n}'\subseteq\mathcal{C}_{2n}$. The proof is then complete. \qed\\


\begin{thebibliography}{ASM}

\bibitem[ABS11]{ABS11} F.~Ardila, T.~Bliem, and D.~Salazar,
\emph{Gelfand-Tsetlin polytopes and Feigin-Fourier-Littelmann-Vinberg polytopes as marked poset polytopes.} J. Combin. Theory Ser. A 118 (2011), no. 8, 2454--2462.

\bibitem[BFF17]{BFF17} T. Backhaus, X. Fang and G. Fourier. \emph{Degree cones and monomial bases of Lie algebras and quantum groups.} Glasgow Mathematical Journal 59, no. 3 (2017): 595-621.

\bibitem[Bal22]{B20} G. Balla. \emph{Symplectic PBW degenerate flag varieties; PBW tableaux and defining equations.} Transformation Groups (2022). \url{https://doi.org/10.1007/s00031-022-09725-9}

\bibitem[BO21]{BO21} G. Balla, and J. A. Olarte. \emph{The Tropical Symplectic Grassmannian.} International Mathematics Research Notices 2023, no. 2 (2023): 1036–1072.

\bibitem[BG84]{BG84}
R. Bieri, and J.R.J. Groves. \emph{The geometry of the set of characters induced by valuations.} J. Reine Angew. Math. 347 (1984), 168--195.

\bibitem[BCI21]{BCI21} M. Boos, G. Cerulli Irelli, \emph{On degenerations and extensions of symplectic and orthogonal quiver representations,} preprint, arXiv:2106.08666v1.

\bibitem[BLMM17]{BLMM17} L. Bossinger, S. Lambogila, K. Mincheva and F. Mohammadi. \emph{Computing toric degenerations of flag varieties.} In Combinatorial algebraic geometry, pp. 247-281. Springer, New York, NY, 2017.

\bibitem[BEZ21]{BEZ21} M. Brandt, C. Eur, L. Zhang. \emph{Tropical flag varieties}, Adv. Math., Volume 384, June 2021, 107695.

\bibitem[CFF+17]{CFFFR17} G. Cerulli Irelli, X. Fang, E. Feigin, G. Fourier and M. Reineke. \emph{Linear degenerations of flag varieties,} Mathematische Zeitschrift, (2017) 287: 615--654.

\bibitem[CFF+20]{CFFFR20} G. Cerulli Irelli, X. Fang, E. Feigin, G. Fourier and M. Reineke. \emph{Linear degenerations of flag varieties: partial flags, defining equations, and group actions.} Mathematische Zeitschrift, 2020, 296(1), 453--477.

\bibitem[Dec79]{Dec79} C. De Concini. \emph{Symplectic standard tableaux}. Adv. Math., 34(1) (1979), 1--27.

\bibitem[FFR16]{FFR16} X. Fang, G. Fourier, and M. Reineke. \emph{PBW-type filtration on quantum groups of type ${\tt A}_n$.} Journal of Algebra 449 (2016): 321--345.

\bibitem[FFL17]{FFL17} X. Fang, G. Fourier, and P. Littelmann. \emph{Essential bases and toric degenerations arising from birational sequences}, Adv. Math, Volume 312, May 2017, Pages 107--149.

\bibitem[FFR21]{FFR21} X. Fang, G. Fourier, and M. Reineke. \emph{Cones from quantum groups to tropical flag varieties.} Journal of Algebraic Combinatorics, 2021, 54(3), 815--835.

\bibitem[FFFM19]{FFFM19} X. Fang, E. Feigin, G. Fourier, and I. Makhlin. \emph{Weighted PBW degenerations and tropical flag varieties.} Communications in Contemporary Mathematics 21, no. 01 (2019): 1850016.

\bibitem[FG22]{FG22} X. Fang, M. Gorsky, \emph{Exact structures and degeneration of Hall algebras}, Adv. Math, Volume 398, March 2022, 108210.

\bibitem[FFL11a]{FFL11a} E.~Feigin, G.~Fourier and P.~Littelmann. \emph{PBW filtration and bases for irreducible modules in type ${\tt A}_n$}. Transform. Groups 16 (2011), no. 1, 71--89.

\bibitem[FFL11b]{FFL11b} E.~Feigin, G.~Fourier, and P.~Littelmann. \emph{PBW Filtration and Bases for symplectic Lie Algebras}, Internat. Math. Res. Notices (2011), Vol. 2011, 5760--5784.

\bibitem[Fei12]{Fei12} E. Feigin.  \emph{$\mathbb{G}^a_M$ Degenerations of flag varieties}, Selecta Mathematica 18, no. 3 (2012): 513--537.

\bibitem[FFL14]{FFL14} E. Feigin, M. Finkelberg and P. Littelmann. \emph{Symplectic degenerate flag varieties}. Canadian Journal of Mathematics, 66(6) (2014), 1250--1286.

\bibitem[FeFL17]{FeFL17} E. Feigin, G. Fourier, and P. Littelmann. \emph{Favourable modules: filtrations, polytopes, Newton–Okounkov bodies and flat degenerations.} Transformation Groups 22.2 (2017): 321--352.

\bibitem[Ful93]{Ful93} W. Fulton. \emph{Introduction to toric varieties}. Vol. 131. Princeton university press, 1993.

\bibitem[Har13]{Har13} R. Hartshorne. \emph{Algebraic geometry}. Vol. 52. Springer Science \& Business Media, 2013.

\bibitem[Hum12]{Hum12} J. E. Humphreys. \emph{Introduction to Lie algebras and representation theory.} Vol. 9. Springer Science \& Business Media, 2012.

\bibitem[MS15]{MS}
D.~Maclagan, B.~Sturmfels, \emph{Introduction to tropical geometry}. Graduate Studies in Mathematics, 161. American Mathematical Society, Providence, RI, 2015. xii+363 pp.

\bibitem[Mak22]{Mak22} I. Makhlin. \emph{Gelfand–Tsetlin degenerations of representations and flag varieties.} Transformation Groups 27, no. 2 (2022): 563--596.

\bibitem[SS04]{SS04} D. Speyer, and B. Sturmfels. \emph{The tropical Grassmannian.} Adv. Geom 4 (2004): 389--411.

\bibitem[Sta86]{Sta86} R.~P. Stanley, \emph{Two poset polytopes}, Discrete Comput. Geom., 1(1):9--23, 1986.

\end{thebibliography}
\end{document}